\documentclass[12pt, a4paper,leqno,oneside]{amschanged}

\usepackage{tikz}
\usepackage{calc}

\usepackage{xcolor}

\usepackage{hyperref}

\usepackage{amsmath}
\usepackage{amssymb}
\usepackage{amsthm}
\usepackage[foot]{amsaddr}

\usepackage{atbeginend}

\usepackage{psfrag}
\usepackage{graphicx}
\usepackage{hyperref}
\hypersetup{
    colorlinks,%
    citecolor=black,%
    filecolor=black,%
    linkcolor=black,%
    urlcolor=black
}

\usepackage{wasysym}
\usepackage{fullpage}

\usepackage{tikz}

\usepackage{type1cm}

\newtheorem{theorem}{Theorem}[section]

\newtheorem{proposition}[theorem]{Proposition}
\newtheorem{lemma}[theorem]{Lemma}

\numberwithin{equation}{section}

\theoremstyle{remark}
\newtheorem{remark}[theorem]{Remark}

\def\Max{{\mathrm{max}}}

\def\cD{\mathcal D}

\def\Z{\mathbb{Z}}

\usepackage{titlesec}
\titleformat{\section}{\Large\bfseries}{\thesection}{1em}{}
\titleformat{\subsection}{\bfseries}{\thesubsection}{1em}{}

\renewcommand{\theenumi}{\arabic{enumi}}
\renewcommand{\labelenumi}{\theenumi)}

\usepackage{courier}

\usepackage{array}
\newcolumntype{e}{>{\displaystyle}r @{\,} >{\displaystyle}c @{\,} >{\displaystyle}l}

  \newcounter{constant}
  \newcommand{\newconstant}[1]{\refstepcounter{constant}\label{#1}}
  \newcommand{\useconstant}[1]{c_{\textnormal{\tiny \ref{#1}}}}
\def\clap#1{\hbox to 0pt{\hss#1\hss}}

\def\mathclap{\mathpalette\mathclapinternal}

\def\mathclapinternal#1#2{\clap{$\mathsurround=0pt#1{#2}$}}

\usepackage{calc}
\def\arraypar#1{\parbox[c]{\textwidth - 2cm}{\centering #1}}

\begin{document}

\fontsize{12}{14}\rm
\addtolength{\abovedisplayskip}{.5mm}
\addtolength{\belowdisplayskip}{.5mm}
\AfterBegin{enumerate}{\addtolength{\itemsep}{2mm}}


\title{\LARGE \usefont{T1}{tnr}{b}{n} \selectfont
  I\lowercase{nterface motion in random media}} 

\author{\normalsize \itshape T. B\lowercase{odineau} $^1$ \color{white} \tiny and}
\address{$^1$ \'Ecole Normale Sup\'erieure, D\'epartement de
  Math\'ematiques et Applications, \newline \hspace*{10mm} Paris 75230,
  France {\itshape \texttt{bodineau@dma.ens.fr}}, {\itshape \texttt{augusto.teixeira@ens.fr}}.}

\author{\color{black} \normalsize \itshape A. T\lowercase{eixeira} $^1$ $^2$}
\address{$^2$ Instituto Nacional de Matem\'atica Pura e Aplicada, \newline
  \hspace*{10mm} Rio de Janeiro 22460-320, Brazil {\itshape \texttt{augusto@impa.br}}.}


\date{\today}


\begin{abstract}
We study the pinning phase transition for discrete surface dynamics in random environments.
A renormalization procedure is devised to prove that the interface moves with positive velocity under a finite size condition.
This condition is then checked for different examples of microscopic dynamics to illustrate the flexibility of the method.
We show in our examples the existence of a phase transition for various models, including high dimensional interfaces, dependent environments and environments with arbitrarily deep obstacles.
Finally, our ballisticity criterion is proved to be valid up to the critical threshold for a Lipschitz interface model.
\end{abstract}

\maketitle

\section{Introduction}
\label{s:intro}

Driven interfaces in random media have been studied in several physical contexts motivated for example by sliding charge waves \cite{fisher} and fluid flow in random media \cite{KoplikLevine}.
A simple way to state the problem in the latter framework, is to imagine a droplet of water on a random rough substrate. Tilting slowly the substrate induces a force leading to the motion of the droplet above a threshold tilt angle.
To model this transition, a crude approximation is to suppose that locally the contact line of the droplet on the substrate is straight and to parametrize it by a function $u(x,t)$ where the wetted substrate is the area $\{ (x,y) \in \mathbb{R}^2,  y \leq  u(x,t) \}$ and the dry substrate $\{ (x,y) \in \mathbb{R}^2,  y >  u(x,t) \}$, see Figure~\ref{fig: droplet}. The motion of the interface is then approximated by the semi-linear parabolic PDE
\begin{align}
\label{eq: continu}
\forall x \in \mathbb{R}, \qquad
& \partial_t u(x,t,\omega) = \Delta u(x,t, \omega) + f(x,u(x,t, \omega), \omega) + F \\
& u(x,0, \omega) = 0
\end{align}
where $f(x,y,\omega) + F$ is the random force on the interface at position $(x,y)$ and $f(x,y,\omega) \in \mathbb{R}$ is of mean zero. The mean force $F >0$ is induced by the tilt and modulated by the disorder of the rough substrate $f(x,y,\omega)$, where $\omega$ is a realization of a given substrate. It is also interesting to consider $x \in
\mathbb{R}^{d-1}$ with $d \geq 3$, for example to model fluid displacement in porous media.

\begin{figure}[ht]
\centerline{
    \includegraphics[angle=0, width=0.35\textwidth]{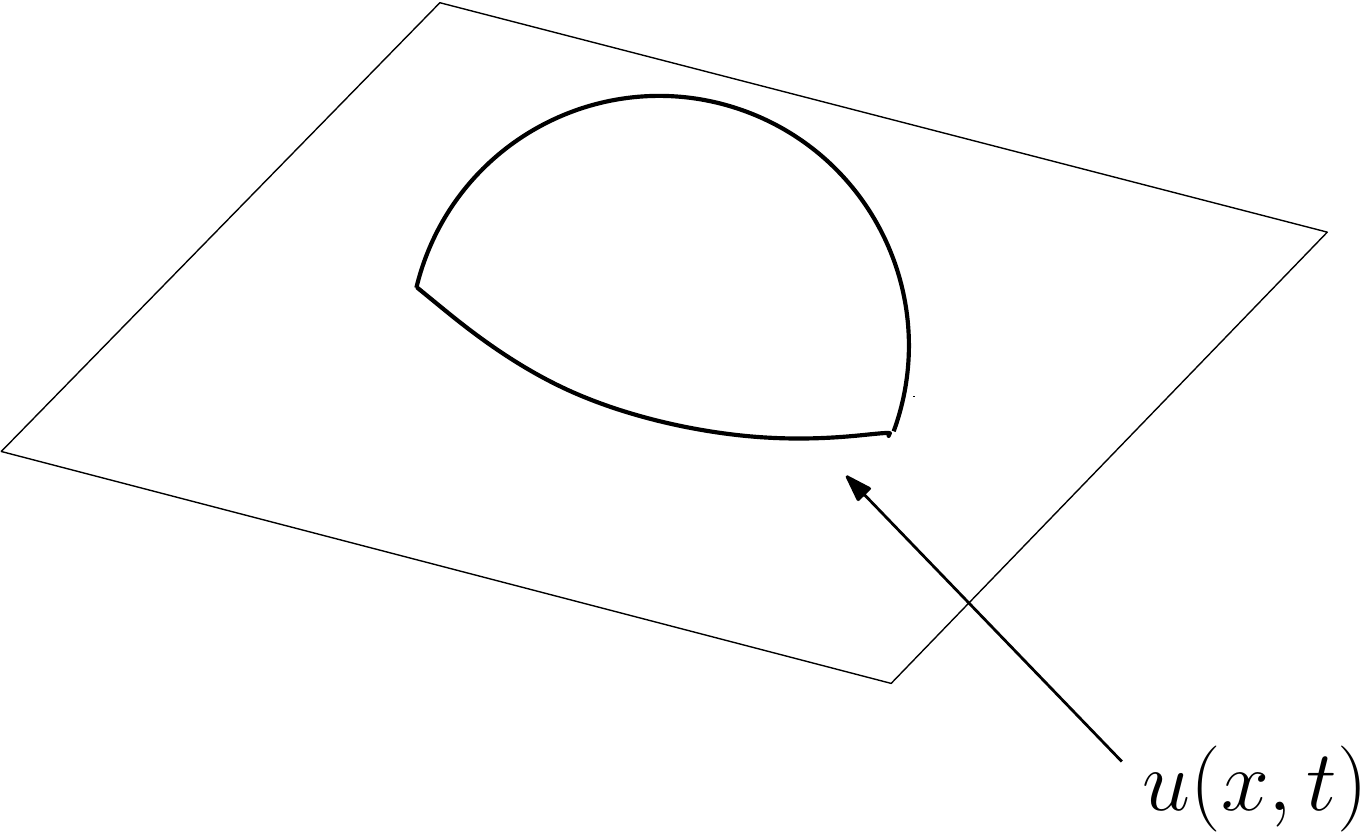}
    \hskip1cm
    \includegraphics[angle=0, width=0.35\textwidth]{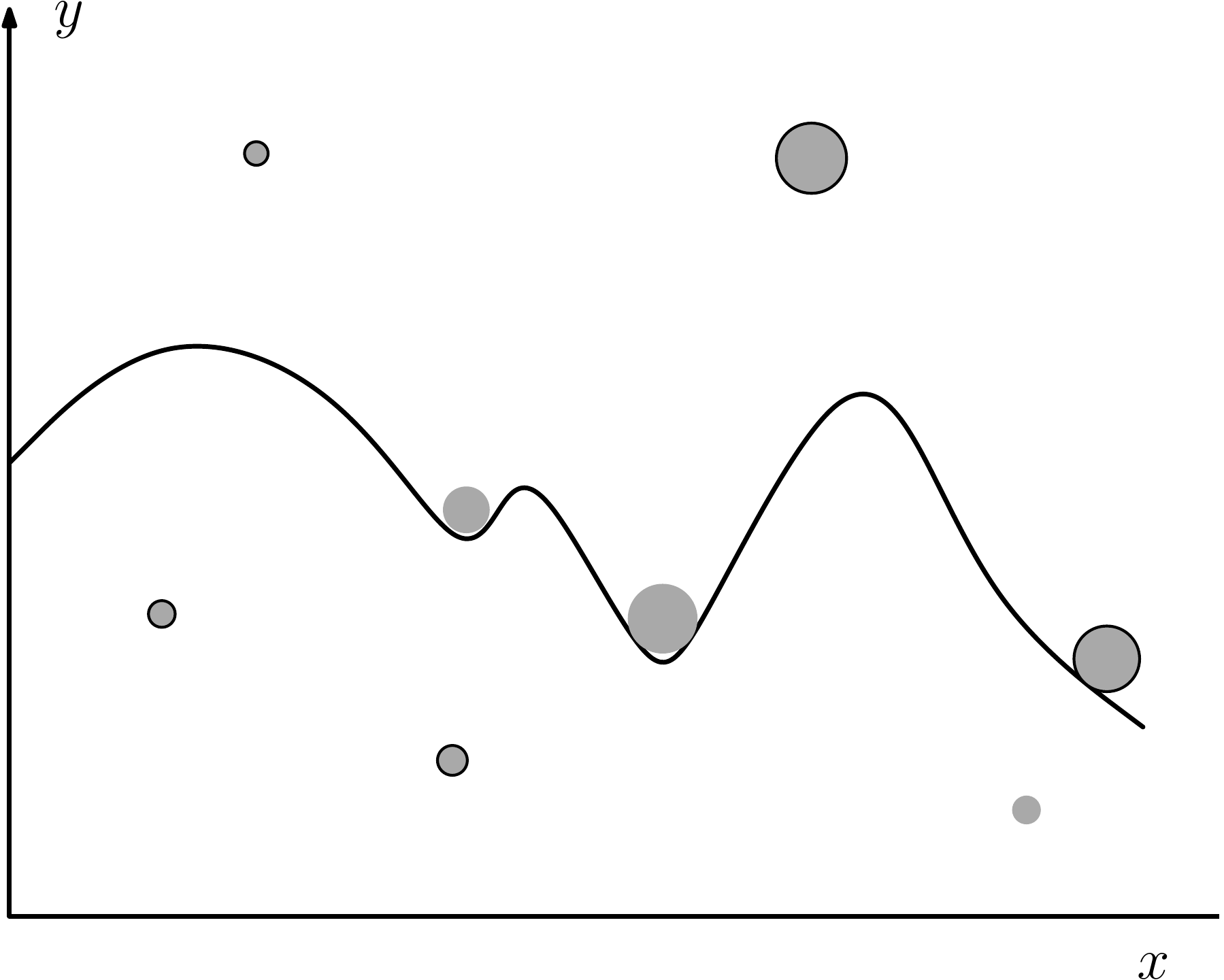}}
    \caption{\small On the left, a droplet on a substrate is depicted.
  Locally,  the contact line is approximated by a function $x \to u(x,t)$ (see the right picture).
     The disorder media is represented by randomly scattered obstacles with different strengths (the gray circles). The interface evolves in the $y$-axis direction and has to overcome the obstacles which are preventing the evolution.  }
    \label{fig: droplet}
\end{figure}

From toy models \cite{TL92}, mean field approximation \cite{Vannimenus} and functional renormalization group analysis \cite{NarayanFisher, doussal}, physicists understood that the interface \eqref{eq: continu} undergoes a phase transition from a pinned regime to a positive speed regime, depending on a critical force threshold
\begin{equation}
  F_c = \inf \Big\{F > 0; \qquad  \liminf_t \frac{u(0,t,\omega)}{t} > 0 \text{ a.s. under the external force $F$}\Big\}.
\end{equation}
The value of the critical exponents at the depinning  transition remains a  controversial issue. From a mathematical point of view, the behavior of the system around $F_c$ is not well understood yet, but the perturbative regimes have already been studied.
For $d \geq 2$ and when the mean force $F>0$ is close to zero, the pinning induced by the disorder is strong enough to prevent the interface from moving \cite{DDS11} ($F_c > 0$). Key to this proof is the occurrence of a Lipschitz percolation membrane \cite{DDGHS}, which blocks the interface's motion. On the other hand when $d =2$ and $F$ is large enough, it is proven in \cite{CovilleDirrLuckhaus, DS12} that the interface moves with a speed bounded from below ($F_c < \infty$). Note that these techniques are restricted to some particular models in $d = 2$ for independent environments and very strong forces. One of the goals of this paper is to provide tools that do not impose these various restrictions.

Even though the driving force is positive on average, there are rare random locations where the pinning force can be arbitrarily large, leading to a competition between these pinning obstacles and the elastic force, which tends to straighten the interface.
This competition occurs on multiple scales, involving collective effects and gets more complicated at the vicinity of the transition. Thus, it is not clear if the phase transition is sharp in the following sense. Let
\begin{equation}
  \label{e:F_t}
  F_d = \inf \{ F > 0; \qquad \lim_t u(0,t,\omega) = \infty \text{ a.s. under $F$} \},
\end{equation}
which is a weaker definition of depinning transition that does not require positive speed of divergence. An important question now is whether $F_d = F_c$, or in other words, is there an intermediate regime, where the surface is transient but sub-ballistic? It is well known that this is the case for the one-dimensional random walk in a random environment, see for instance \cite{zeitouni}. Another goal of ours is to investigate this issue and to provide tools to study the vicinity of the critical threshold.

\medskip

In rigorous equilibrium statistical mechanics, these sharp analysis often rely on a two-steps procedure. First, properties of the model are derived under some criterion and then the validity of the criterion is extended up to
the critical point. For example in percolation, the assumption of finite expected cluster size in the subcritical regime \cite{MR0101564} or the percolation on slabs for the supercritical phase. This allows one to implement a coarse graining procedure, leading to strong statements for the percolation model \cite{AizenmanChayesRusso, Pisztora}. Second, these assumptions are established for a range of parameters as large as possible and ideally up to the critical temperature. This is the case in percolation for the exponential decay of correlations \cite{Menshikov, AizenmanBarsky} and the slab percolation threshold \cite{GrimmettMarstrand}.

In this paper, we consider several discrete interface dynamics similar to \eqref{eq: continu}. Our discrete interface is modeled as a function $S_t:\mathbb{Z}^{d-1} \to \mathbb{Z}$ which evolves on discrete time $t = 0,1,\dots$, according to \eqref{e:evolve} below. In \eqref{e:W} we introduce a finite size criterion, which quantifies the probability that the interface $S_t$ gets completely blocked on a finite size box $C^L$ (of side length proportional to $L$):
\begin{equation}
  \label{e:W_intro}
  \mathbb{P}[\text{the surface $S$ gets completely blocked in $C^L$}] \text{ decays polynomially in $L$,}
\end{equation}
see \eqref{e:W} for more precise statement.

We don't require our random environment to be composed of i.i.d. random variables, since our techniques are robust to the presence of dependence. We however need a replacement for the independence of the environment, namely the mixing condition \eqref{e:cov2}.

Our main result, Theorem~\ref{c:speed}, states that
\begin{equation}
  \label{e:speed_intro}
  \text{under conditions \eqref{e:W_intro} and \eqref{e:cov2}, we have } \liminf_t \frac{S_t}{t} > 0.
\end{equation}
In other words, the theorem reduces the question of ballisticity ($F_c < \infty$) to a mixing condition together with a finite-size and static criterion.

The intuition behind the above result is that, as the interface evolves, deeper traps are encountered in various scales which locally block the interface. The criterion \eqref{e:W_intro} allows us to evaluate the frequency of these traps and Theorem~\ref{c:speed} proves by a multi-scale analysis that the traps may lower the speed of the interface but without harming ballisticity, thanks to the elastic force.

Let us first observe that our result works for any dimension $d \geq 2$. As an application, we prove the existence of a phase transition between ballistic and blocked regimes for the evolution of a $(d-1)$-dimensional interface in $\mathbb{Z}^d$, see Lemma~\ref{l:blocklips}.

We also conjecture that for several models, our criterion \eqref{e:W_intro} should hold up to the critical threshold (i.e. for any fixed force $F > F_c$). We have been able to establish this for a specific two-dimensional model in Lemma \ref{l:unblockd=2}, implying that $F_c = F_d$.

\medskip

The proof of our main result relies on multi-scale renormalization, which is a classical tool in statistical mechanics. In the physics literature, this technique is often referred to as the renormalization group and is very successful in establishing critical and near-critical properties of physical systems, see \cite{Fis96} for an overview on this subject. In the mathematical literature multi-scale analysis has been more useful for studying perturbative regimes, see for instance \cite{Szn09}, \cite{SV05} and \cite{MR841669} for some examples of such applications. We introduce several novelties in our renormalization scheme and in particular we devise a rigorous coarse graining argument for interface motion.

Finally we would like to stress two further advantages of the multi-scale approach that we introduce for this problem. The first is that it allows us to study dependent environments, since this analysis enters very little into the microscopic details of the dynamics, see Remark~\ref{r:constants}. Also, we were able to reduce the problem of ballisticity of the surface to the static criterion \eqref{e:W_intro}, which often translates into a percolation condition where several known techniques may be useful,
see for example Lemma~\ref{l:unblockd=2}. We don't see any obstruction to adapting our methods to different contexts, such as continuous time interface dynamics.

\medskip

The first part of the paper is devoted to the renormalization procedure under general assumptions.
In the second part, these assumptions are checked for different models.

{\bf Acknowledgments} - We would like to thank Vivien Lecomte for very useful discussions and insights on the physics of interfaces and Marcelo Hil\'ario for helpful discussions and careful reading. T.B Acknowledges the support of the french ministry of education through the ANR 2010 BLAN 0108 01 grant. A.T. is also grateful to the Brazilian-French Network in Mathematics for the opportunity to visit Paris during the elaboration of this work and for the financial support from CNPq, grants 306348/2012-8 and 478577/2012-5.

\section{Model}
\label{s:notation}

Fix from now on a dimension $d \geq 2$ and denote each element of $\Z^d$ by $(x,z)$, where $x \in \Z^{d-1}$ and $z \in \Z$. Throughout this text, we are going to consider an energy landscape on $\mathbb{Z}^d$, in which each site $(x,z) \in \mathbb{Z}^d$ is assigned an energy $\omega(x,z) \in \overline{\mathbb{R}} := \mathbb{R} \cup \{-\infty\}$.
This energy landscape will be random, or in other words, we are given a probability space $(\Omega, \mathcal{A}, \mathbb{P})$, where
\begin{equation}
  \label{e:Omega}
  \begin{array}{c}
    \Omega = \overline{\mathbb{R}}^{\mathbb{Z}^d} \text{ is the space of energy landscapes
      on $\mathbb{Z}^d$},\\
    \mathcal{A} \text{ is the product $\sigma$-algebra on $\Omega$ and}\\
    \mathbb{P} \text{ is a probability measure on $\mathcal{A}$,}
  \end{array}
\end{equation}
providing the law of the environment.
For simplicity, we will always assume that
\begin{equation}
  \label{e:invariant}
  \text{the law $\mathbb{P}$ is invariant with respect to translations of $\mathbb{Z}^d$.}
\end{equation}

Another important assumption we make on the law $\mathbb{P}$ concerns its long range dependence. Intuitively speaking, we want that, if two events are determined by what happens in far away regions of the space, then they should be roughly independent. This is made precise in the following assumption.

Suppose that there exist $\alpha > 0$ and $\newconstant{c:cov}\useconstant{c:cov} > 0$ such that, for all $L \geq \useconstant{c:cov}$, given any pair of cubes $C_1$ and $C_2 \subseteq \mathbb{Z}^d$ with side length $3L$, within distance at least $L$ from one another,
\begin{equation}
  \label{e:cov}
  \begin{array}{c}
  \medskip
  \text{for any functions $f_1:\mathbb{R}^{C_1} \to [0,1]$ and $f_2:\mathbb{R}^{C_2} \to [0,1]$,}\\
  \textnormal{Cov}(f_1,f_2) \leq L^{-\alpha}.
  \end{array}
\end{equation}
Let us observe that in particular this implies that the environment $\omega$
is mixing, therefore ergodic, under translations of $\mathbb{Z}^d$.
More generally, we will need the assumption below, dealing with $D$ boxes instead of two
\begin{equation}
  \label{e:cov2}
  \begin{array}{c}
  \medskip
  \text{for any functions $f_i: \mathbb{R}^{C_i} \to [0,1]$, for $i = 1, \dots, D$,}\\
  E[f_1 \cdot \dots \cdot f_D] \leq E[f_1] \dots, E[f_D] + L^{-\alpha}.
  \end{array}
\end{equation}
Where we again assume that the boxes $C_i$ have side length $3L$ and mutual distance at least $L$ from each other.
Note that an assumption of the form \eqref{e:cov2} is much weaker than the typical restriction on the environment to be independent.

Given an environment $\omega \in \Omega$, we are going to consider a
deterministic motion of a surface on $\mathbb{Z}^d$, subject to the
potential $\omega$.  To make this more precise, let us introduce some
definitions.

We write $\overline{\mathbb{Z}} = \mathbb{Z} \cup \{-\infty\}$ and call any given function $S:\mathbb{Z}^{d-1} \to \overline{\mathbb{Z}}$ a \textit{surface}.
Note that we do not impose any regularity conditions on $S$.
Intuitively speaking, we will regard the graph of $S$ (which is a subset of $\mathbb{Z}^{d-1} \times \bar{\mathbb{Z}}$) as a physical surface moving through the environment $\omega$.

Let us introduce the discrete partial derivatives of an interface $S:\Z^{d-1} \to \Z$. Define the set of canonical directions $E = \{\pm e_k; k = 1, \dots, d-1 \}$. Given $x \in \Z^{d-1}$ and $e \in E$, the discrete derivative of $S$ at $x$ in direction $e$ is given by $\partial_e S(x) = S(x + e) - S(x)$.


Given an initial surface $S_0$, an environment $\omega \in \Omega$ and an
 \textit{update function}
$F:\mathbb{Z}^E \times \overline{\mathbb{R}} \to \{0,1\}$,
the evolution of the surface $S_t$ is determined as follows.
Suppose that $S_t$ has been defined for some integer time $t \geq 0$. Then,
\begin{equation}
  \label{e:evolve}
  S_{t+1}(x) = S_t(x) + F \Big(\partial_{e_1} S_t(x), \partial_{-e_1} S_t(x), \dots, \partial_{e_{d-1}} S_t(x), \partial_{-e_{d-1}} S_t(x); \omega(x,S_t(x)) \Big).
\end{equation}
Intuitively speaking, $F$ reads the derivatives of $S_t$ at $x$ in all directions and the current potential at $(x, S_t(x))$, then it
returns one if the surface is supposed to move in the position $x$ at time
$t+1$ and zero otherwise.
Let us convention that $-\infty + 1 = -\infty$ and that $F$ equals zero if any partial derivative is equal to $-\infty$ (recall that these values are also allowed to occur in $S$). Observe that there is no need to evaluate $\omega(x, -\infty)$, since no update is relevant if $S(x) = -\infty$.

Although our techniques can easily deal with the general update functions $F$  defined above, some of the most interesting situations involve a function $F$ solely depending on the discrete Laplacian of the surface $S: \Z^{d-1} \to \Z$
\begin{equation}
  \label{e:Laplacian}
  \Delta S(x) = \sum_{k = 1,\dots, d-1} S(x + e_k) - 2 S(x) + S(x - e_k).
\end{equation}
In such situations, the notation for the evolution rule can be simplified to the following
\begin{equation}
  \label{e:SLaplace}
  S_{t+1}(x) = S_t(x) + F\Big(\Delta S_t (x), \omega(x, S_t(x))\Big),
\end{equation}
where $F:\Z \times \overline{R} \to \{0,1\}$ is the function governing this simplified update rule.

We would like to stress a few consequences of the rule \eqref{e:evolve}:
\begin{align}
  \label{e:determ}
  &\text{the whole evolution of the initial surface $S_0$
  depends solely on $\omega$,}\\
  \nonumber
  &\text{not involving any additional randomness,}\\
  \label{e:simult}
  &\text{the surface moves simultaneously at every location where it is
    allowed to,}\\
  \label{e:onebyone}
  &\text{$S_t(x)$ and $S_{t+1}(x)$ differ by at most one for every $t \geq
    0$,}\\
  \label{e:homogen}
  &\text{our evolution rule is translation invariant on $\mathbb{Z}^d$.}
\end{align}
We chose the exact evolution mechanism described in \eqref{e:evolve} in order to simplify our proofs.
However, we believe that the techniques developed in this paper could be extended to other problems involving for instance a random update function.

Let us now make  further assumptions on the function $F$ determining
the dynamics above. Suppose from now on that
\begin{align}
  \label{e:Fmonotone}
  &\text{$F$ is non-decreasing on all its coordinates.}\\
  \label{e:Fpinned}
  &\text{$F(a_1, \dots, a_{2(d-1)}, -\infty) = 0$ for all values of $a_1, \dots, a_{2(d-1)}$,}
\end{align}
Assumption \eqref{e:Fmonotone} intuitively says that the smaller the energy $\omega(x, S_t(x))$ and the lower the position of the surface at neighboring sites, the harder it will be for $S_t(x)$ to advance. Let us emphasize that \eqref{e:onebyone} and \eqref{e:Fmonotone} imply the attractiveness of our dynamics, i.e. $S_t \leq S_t'$ for every $t \geq 1$, whenever $S_0 \leq S_0'$.
Also, assumption \eqref{e:Fpinned} states that an infinitely deep well can never be overcome, pinning the surface forever on that location.

\bigskip

We are going now to state the finite size criterion which will be used later in our renormalization approach.
First, we need to define, for integers $h, L \geq 1$, the boxes
\begin{equation}
\label{e:CB}
\begin{array}{c}
C^{h,L} = [0,3hL)^{d-1} \times [0, H_L) \cap \mathbb{Z}^d \text{ and}\\
B^{h,L} = [hL,2hL)^{d-1} \times [0,L) \cap \mathbb{Z}^d,
\end{array}
\end{equation}
where $H_L = L^{a+1}$ with $a$ chosen to be integer for notational simplicity. See Figure~\ref{f:boxes}.

\begin{figure}[ht]
  \begin{center}
    \includegraphics[angle=0, width=0.5\textwidth]{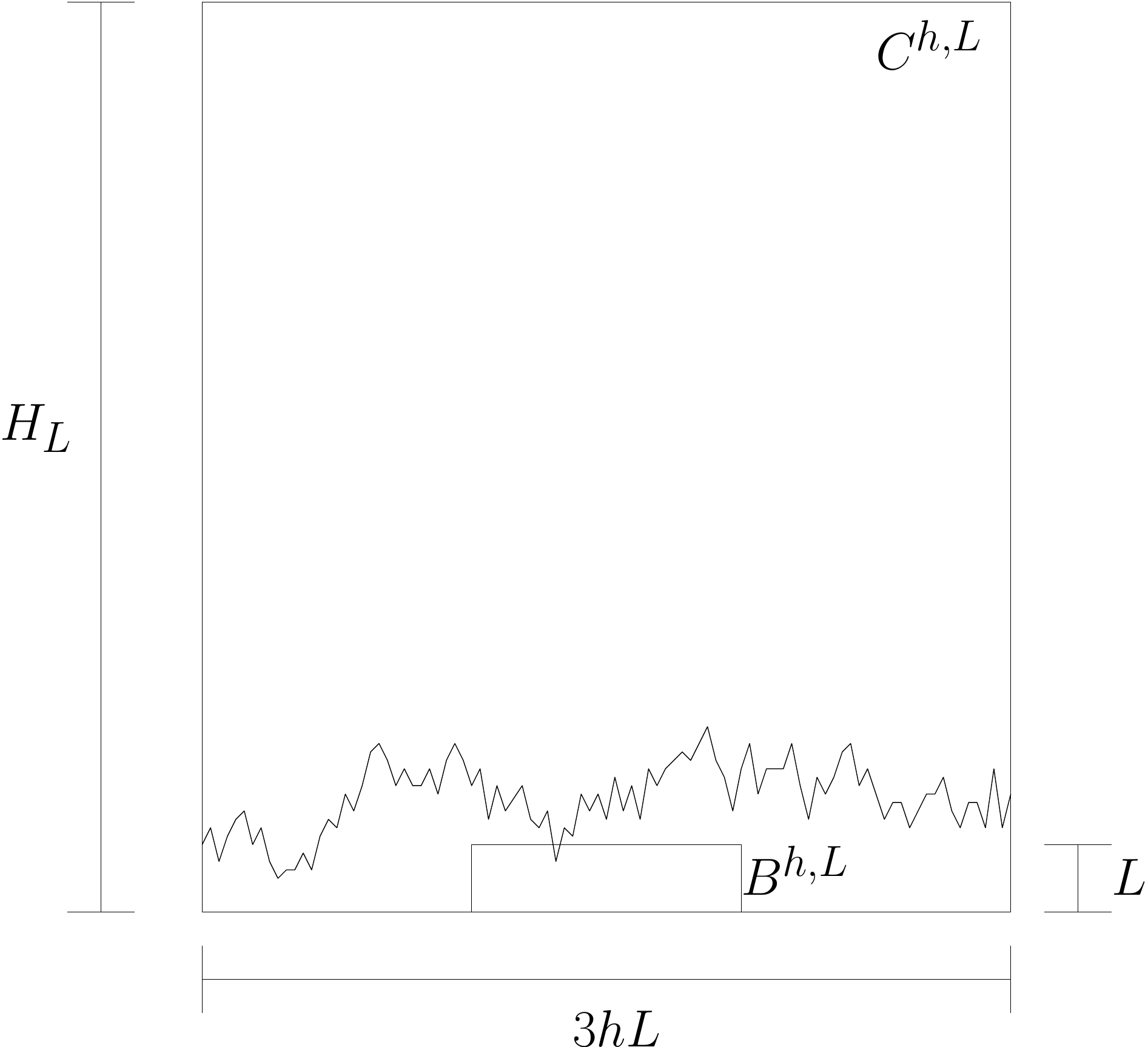}\\
    \caption{\small The boxes $B^{h,L}$ and $C^{h,L}$ appearing in \eqref{e:CB} (for $d = 2$) and a blocking surface as in hypothesis \eqref{e:W}.
Note that the typical box height will be chosen such that $H_L \gg L$ (contrary to what is depicted).
    }
    \label{f:boxes}
  \end{center}
\end{figure}

We will also consider interface evolutions restricted to  the domain $C^{h,L}$ by forcing the environment $\omega$ to be infinite outside  $C^{h,L}$.
We say that a given surface $S_t$ is obstructed at $x$ if $S_t(x)$ does not change under one step of our update rule \eqref{e:evolve}, i.e. $S_{t+1}(x) = S_t(x)$.
Given $h$ and $L$ as in \eqref{e:CB}, we say that a surface $S$ blocks $C^{h,L}$ if the following happens:
\begin{itemize}
 \item the graph of $S$ intersects $B^{h,L}$,
 \item $S(x)$ is non-negative for every $x$ in $[0, 3hL)^{d-1}$ and
 \item for every $x \in [0, 3hL)^{d-1}$ such that $S(x) \leq H_L$, $S$ is obstructed in $x$,
\end{itemize}
see Figure~\ref{f:boxes}.

We would like to emphasize that the existence of a surface blocking $C^{h,L}$ is an event that can be determined by solely analyzing the configuration of the environment inside $C^{h,L}$. This localization property will be used in conjunction with the decoupling bound  \eqref{e:cov2}.

In the unpinned regime, one expects that the effective force on the interface is positive and that
the typical interfaces, when the evolution is restricted to $C^{h,L}$,   look like a bell with maximal height in the middle of the block.
Interfaces blocking $C^{h,L}$ should be rare in the unpinned regime.
The details of their definition
were designed in order to balance between being easy to prove in examples, see Section~\ref{s:perturb} and being strong enough to imply our main result, see \eqref{e:vkdecay}.


\bigskip

We are now in position to state our criterion \newconstant{c:unpinning}

\medskip

\noindent
{\bf Finite size criterion.}
Suppose that
\begin{equation}
 \label{e:W}
 \begin{array}{c}
   \text{there exist $\useconstant{c:unpinning}, \rho > 0$ and an integer $h \geq 2$, such that, for every $L \geq \useconstant{c:unpinning}$,}\\
   \mathbb{P}\big[ \text{there exists a surface $S$ blocking $C^{h,L}$} \big] \leq L^{-\rho}.\\
 \end{array}
\end{equation}

\begin{remark}

Under various scenarios, we expect the decay in \eqref{e:W} to be exponential and refer to Section~\ref{s:perturb} for some examples.
In dimension two, the finite size criterion resembles a percolation statement on decay of connectivity, therefore it could, in principle, be proved using percolation techniques.
For the case $d \geq 3$, we can still relate condition \eqref{e:W} to surface percolation such as studied in \cite{DDGHS}, see Section~\ref{s:perturb} for details.
\end{remark}

The ballistic evolution of the interface is a consequence of this criterion
\begin{theorem}
\label{c:speed}
Let $S_0$ be the flat surface given by $S_0(x) = 0$ for every $x \in \mathbb{Z}$. Then, using that \eqref{e:cov2} and \eqref{e:W} hold with $\alpha$, $\rho$ and $D$ satisfying \eqref{e:assumpt1}--\eqref{e:assumpt2}, there exists a constant \newconstant{c:speedbound} $\useconstant{c:speedbound}(\alpha, d, D, h, \gamma, \useconstant{c:cov}, \useconstant{c:unpinning}) > 0$ such that
\begin{equation}
 \label{e:speed}
 \liminf_{t \to \infty} \frac{S_t(0)}t \geq \useconstant{c:speedbound}, \text{ $\mathbb{P}$-a.s.}
\end{equation}
\end{theorem}

See Remark~\ref{r:constants} on the validity of conditions \eqref{e:assumpt1}--\eqref{e:assumpt2}.

\section{Renormalization}
\label{s:renorm}

Fix an integer $\gamma$ such that $\gamma > d + 2 a \geq 2$ (recall that the height of $C^{h,L}$ is $H_L = L^{1+a}$)
 and let
\begin{gather}
 \label{e:Lk}
 L_0 \geq 100, \quad  \text{ and} \quad
 L_{k+1} = L_k^\gamma, \quad \text{ for $k \geq 0$.}
\end{gather}

Our renormalization scheme will be defined in terms of a collection of boxes (resembling the ones appearing in \eqref{e:CB}).
We first define the sets
\begin{equation}
  \label{e:Mk}
  M_k = \{k\} \times \mathbb{Z}^d, \text{ for $k = 0,1,\dots$}
\end{equation}
whose elements will label the boxes at a given scale $k$ in our renormalization
scheme.
More precisely, for $m = \big(k, (i,j) \big) \in M_k$ (where $i \in \Z^{d-1}$ and $j \in \Z$), let
\begin{equation}
  \label{e:Bm}
  B_m = [0, h L_k)^{d-1} \times [0, L_k) + (i h L_k, jL_k) \subset \mathbb{Z}^d.
\end{equation}
The collection of boxes $\{B_m\}_{m \in M_k}$ form a disjoint tiling of
the whole lattice $\mathbb{Z}^d$ by rectangles of base length $h L_k$ and height $L_k$. We also need to define, given $m = \big(k, (i,j) \big) \in M_k$,
\begin{equation}
  \label{e:Cm}
  \begin{array}{c}
    C_m = \cup_{m'}B_{m'}, \text{ for all $m' = \big(k, (i+i',j+j') \big)$}\\
    \text{with $i' \in \Z^{d-1}$ with $\lVert i \rVert_\infty \leq 1$ and $j' \in \{0,\dots, L_k^a \}$,}
  \end{array}
\end{equation}
where $L_k^a = H_{L_k} / L_k$.
Note that the collection $\{C_m\}_{m \in M_k}$ is composed of overlapping
boxes of base length $3hL_k$ and height $H_{L_k}$, see Figure~\ref{f:boxes}.
Observe the similarity between the above definitions and \eqref{e:CB}.

As above, we will routinely denote points in $\Z^d$ by $(x,y)$, with $x \in \Z^{d-1}$ and $y \in \Z$.

Let us now start to introduce some definitions related to the dynamics of
the surface on the above boxes. First we define, for a given index $m \in
M_k$, the modified environment $\omega_m$ given by
\begin{equation}
  \label{e:omegam}
  \omega_m(x) =
  \begin{cases}
    \omega(x),  &\qquad \text{ if $x \in C_m$ and}\\
    -\infty,   &\qquad \text{ otherwise.}
  \end{cases}
\end{equation}
It is clear from the above that the evolution of $S_0$ under $\omega_m$ solely
depends on the configuration of $\omega$ inside $C_m$ and also that
$S_t$ evolves slower under the environment $\omega_m$ than under $\omega$, see below \eqref{e:Fpinned}.

For $k \geq 0$ and $m = (k, (i,j)) \in M_k$, we also introduce the surface
\begin{equation}
 \label{e:Sm}
 S^m_0(x) =
 \begin{cases}
  j L_k, \qquad & \text{for $x$ in the basis of $C_m$ (projection of $C_m$ into $\Z^{d-1}$),}\\
  -\infty & \text{otherwise.}
 \end{cases}
\end{equation}
The definition of $S^m_0$ taking value $-\infty$ for $x$ outside the basis of $C_m$ was made in order to slow down its dynamics, so that it will serve as a lower bound to any other surfaces starting above this platform thanks to the monotonicity of the dynamics, see \eqref{e:Fmonotone}. For positive times $t \geq 1$, we define $S^m_t$ using the evolution rule in \eqref{e:evolve}, starting at $S^m_0$ and evolving under the potential $\omega_m$.

In order to bound from below the speed of the surface $S^m_t$, we need to introduce the time when this surface has first crossed the box $B_m$ (under the potential $\omega_m$). More precisely, given $m \in M_k$, let
\begin{equation}
  \label{e:Tm}
  \begin{split}
  T_m(\omega) = \inf \big\{t; &\text{ the surface $S^m_t$ is completely above the box $B_m$}\big\},
  \end{split}
\end{equation}
which is defined to be infinity if the above set is empty. We stress once again that the time
$T_m(\omega)$ only depends on the configuration $\omega$ inside the
box $C_m$, since we restrict our dynamics to $\omega_m$ in the above definition.

Our main objective in what follows it to bound $T_{(k,(0,0))}(\omega)$ from above
with high probability (under $\mathbb{P}$) as $k$ grows. This will
imply that a flat surface moves with positive speed, see Theorem~\ref{c:speed}.

The main ingredient of the proof is to relate events occurring at scale $k+1$ with their corresponding event at scale $k$.
This is done with the help of the following lemma, which bounds
the time $T_m$ in terms of a sum of the maximum times to cross each horizontal level of boxes at the previous scale, see Figure~\ref{f:layers}. Before stating the lemma, let us define, given $m \in M_k$ with $k \geq 1$, the indices of layers
\begin{equation}
  \label{e:Jm}
  J_m = \Big\{j \in \mathbb{Z}; B_{m'} \subseteq B_m, \text{ for some $m' = \big(k-1, (i,j)\big)$}\Big\}.
\end{equation}
Note that $|J_m| = L_{k-1}^{\gamma-1}$, see \eqref{e:Lk}.

\begin{lemma}
\label{l:supT}
For any given environment $\omega \in \Omega$ and index $m \in M_k$ with $k
\geq 1$, we have
\begin{equation}
  \label{e:supT}
  T_m(\omega) \leq \sum_{j \in J_m} \sup_{m' \in M_m^{j}} T_{m'}(\omega_{m'}),
\end{equation}
where $M_m^j = \{m' = \big(k-1, (i,j) \big); i \in \mathbb{Z}^{d-1}, C_{m'} \subseteq C_m \}$.
\end{lemma}

\begin{proof}
Recall the definition of $J_m$ in \eqref{e:Jm}. For this proof, we will need the analogous set
\begin{equation}
 \label{e:Im}
 I_m = \Big\{i \in \mathbb{Z}^{d-1}; B_{m'} \subseteq C_m, \text{ for some $m' = \big(k-1, (i,j)\big)$}\Big\}.
\end{equation}
Write $J_m = \big\{j_o + 1, j_o + 2, \dots, j_o + L_{k-1}^{\gamma - 1}  \big \}$ and $I_m = \big[i_o + 1, i_o + 3  L_{k-1}^{\gamma - 1}  \big]^{d-1} \cup \Z^{d-1}$.
We will now prove the following claim by induction in $n = 0, \dots,  L_{k-1}^{\gamma - 1} $,
\begin{equation}
 \label{e:stairs}
 \begin{array}{c}
 \text{after time ${\textstyle \sum\limits_{l = 1}^{n}} \sup_{m' \in M_m^{j}} T_{m'}(\omega_{m'})$, the surface will have surpassed all boxes}\\
 \text{ $B_{{\displaystyle(}k-1, (j_o + n,i){\displaystyle)}}$, with $i \in \Big[i_o + n + 1, i_o + 3  L_{k-1}^{\gamma - 1}  - n \Big]^{d-1}$,}
 \end{array}
\end{equation}
see Figure~\ref{f:layers}.

\begin{figure}[ht]
  \psfrag{n0}[0][0][1.5][0]{$n=0$}
  \psfrag{n1}[0][0][1.5][0]{$n=1$}
  \psfrag{Bk}[0][0][2][0]{$B_m$}
  \psfrag{Ck}[0][0][2][0]{$C_m$}

\end{figure}

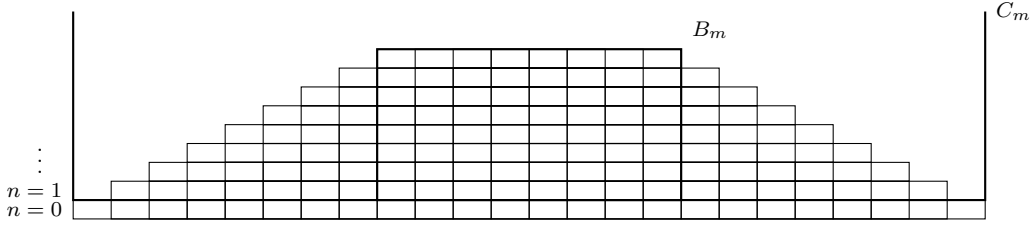
\begin{figure}[ht]
  \label{f:layers}
  \centering
  \begin{tikzpicture}[yscale=.25, xscale=.50]
  \draw[thick] (1, 11) -- (1,1) -- (25,1) -- (25,11);
  \draw[thick] (9, 1) rectangle (17, 9);
  \edef\z{0}
  \foreach \y in {0,...,8} {
    \pgfmathparse{21 - \y}
    \xdef\z{\pgfmathresult}
    \foreach \x in {\y,...,\z} {
      \draw[thin] (\x + 1, \y) rectangle (\x + 4, \y + 1); } }
  \node at (1,.5) [left] {\tiny$n = 0$};
  \node at (1,1.5) [left] {\tiny$n = 1$};
  \node at (.5,3.5) [left] {\tiny$\vdots$};
  \node at (17,9) [above right] {\tiny$B_m$};
  \node at (25,11) [right] {\tiny$C_m$};
  \end{tikzpicture}
    \caption{\small The various layers of boxes indexed in \eqref{e:stairs}. Each horizontal layer corresponds to a value $n$.}
\end{figure}

We start by observing that for $n = 0$ the claim \eqref{e:stairs} is trivially true, since for $n = 0$, \eqref{e:stairs} refers to time zero, when the surface ($S^m_0$) passes in the bottom of $C_m$ (see \eqref{e:Sm}). Therefore, $S^m_0$ is by definition above all $B_{(k-1,(j_o,i))}$ with $i \in I_m$ showing that \eqref{e:stairs} holds for $n = 0$.

Let us now assume that the claim \eqref{e:stairs} holds for $n - 1 <  L_{k-1}^{\gamma - 1}$. Then we have that at time $\tau_{n-1} = {\textstyle \sum_{l = 1}^{n-1}} \sup_{m' \in M_m^{j}} T_{m'}(\omega_{m'})$ the surface $S^m_{\tau_{n-1}}$ is above all the boxes $B_{(k-1, (j_o + n,i))}$, with $i \in [i_o + n, i_o + 3  L_{k-1}^{\gamma - 1}  - (n-1)]^{d-1} \cap \Z^{d-1}$. Then it is not difficult to see that $S^m_{\tau_{n-1}}$ is therefore above all the surfaces $S^{m'}_0$ where
$m'= \big(k-1, (j_0 + n, i) \big)$ with $i = i_o + n + 1, \dots, i_o + 3  L_{k-1}^{\gamma - 1} - n$. Thus, the statement \eqref{e:stairs} for $n$ is a simple consequence of the definition of $T_m$ in \eqref{e:Tm} together with the monotonicity of our dynamics in the initial surface. This proves \eqref{e:stairs}.

To finish the proof of the lemma, one should simply observe that \eqref{e:supT} is implied by \eqref{e:stairs} with $n = L_{k-1}^{\gamma - 1}$.
\end{proof}


We also define recursively, given any $r_0 > 0$ and an integer $D \geq 2$,
\begin{equation}
 \label{e:rk}
 r_{k} =
 r_{k-1} + D (3h)^d \; \tfrac{L_{k-1}^{d+2a}}{L_k}, \text{ for $k \geq 1$.}
\end{equation}
The integer $D$ will be a parameter in the renormalization procedure related to the minimal number of defects to slow down the dynamics restricted to a box (see Lemma \ref{l:cure}).
The sequence $\{ r_k \}$ will bound from below the inverse speed of the surface $S_t$ at scale $k$. More precisely, we will show that for $m \in M_k$, then $T_m \leq r_k L_k$ with high probability as $k$ grows, see Theorem~\ref{t:sdkdecay}. Let us first note that the second term in the above definition is a vanishing error term. Indeed,
\begin{equation}
\frac{L_{k-1}^{d+2a}}{L_{k}} = \frac{L_{k-1}^{d+2a}}{L_{k-1}^{\gamma }} =
L_{k-1}^{-(\gamma-d-2a)}.
\end{equation}
Moreover, the above terms are summable in $k$ (recall that $\gamma > d+ 2a$ and $L_k$ grows super-exponentially). This allows us to conclude that \newconstant{c:boundrk}
\begin{equation}
  \label{e:boundrk}
  \sup_{k \geq 0} r_k \leq \useconstant{c:boundrk}(r_0, L_0, h, \gamma, d, D).
\end{equation}

Given $k \geq 0$ and $m \in M_k$,
\begin{equation}
 \label{e:slowdefect}
 \begin{array}{l}
  \text{we say that the box $C_m$ is slow if $T_m > r_k L_k$.}\\
 \end{array}
\end{equation}
We also define the event and probability corresponding to a box being slow, i.e.
\begin{equation}
  \label{e:slow}
  \begin{array}{c}
    W_m = \{T_m > r_k L_k\} \subseteq \Omega \text{ and}\\
    w_k = \sup_{m \in M_k} \mathbb{P}[W_m] = \mathbb{P}[W_{(k,(0,0))}].
  \end{array}
\end{equation}
Our main objective is to show that (for large enough $L_0$ and $r_0$) $w_k$ decays fast as $k$ grows.

Finally, we need to introduce the notion of blocked boxes, which will resemble the definition below \eqref{e:CB}. For $k \geq 0$ and $m \in M_k$, we say that a given box $C_m$ is blocked if $T_m = \infty$. We define
\begin{equation}
 \label{e:blocked}
 \begin{array}{c}
  V_m = \{T_m = \infty \} \text{ and}\\
  v_k = \sup_{m \in M_k} \mathbb{P} [V_m] \overset{\eqref{e:invariant}}{=} \mathbb{P} [V_{(k,(0,0))}]
 \end{array}
\end{equation}
Let us observe that the finite size criterion \eqref{e:W} implies that, if $L_0 \geq \useconstant{c:unpinning}$,
\begin{equation}
  \label{e:vkdecay}
  v_k \leq L_k^{-\rho}, \text{ for every $k \geq 1$.}
\end{equation}
Indeed, if $T_m = \infty$, the evolution of the surface $S^m$ will eventually halt (meaning that $S^m_{t} = S^m_{s}$, for every $t \geq s$). In this case, we can see that the surface $S^m_s$ blocks $C_m$, so that \eqref{e:vkdecay} is in fact a consequence of \eqref{e:W}.

\bigskip

It is a crucial observation that, for $k \geq 0$ and $m \in M_k$,
\begin{equation}
  \label{e:Tmfinite}
  \text{if $T_m < \infty$, then $T_m$ must be smaller or equal to $(3 h)^{d-1} L_k^{d+a}$}.
\end{equation}
The above statement is a consequence of the evolution mechanism
chosen in \eqref{e:evolve}. Indeed, the surface cannot perform more than
$3^{d-1} (h L_k)^{d-1} H_{L_k}$ updates in the environment $\omega_m$, since $3^{d-1} (h L_k)^{d-1} H_{L_k}$ is the total
number of points in $C_m$. Since $H_{L_k} = L_k^{a+1}$, at time $(3 h)^{d-1} L_k^{d+a}$, the surface $S^m$ will be completely obstructed under $\omega_m$ and it will either have surpassed $B_m$ (in which case $T_m \leq
(3 h)^{d-1} L_k^{d+a}$) or not (implying that $T_m = \infty$). This proves \eqref{e:Tmfinite}.

\medskip

Our renormalization argument is based on the following fact: the only way in which a box at scale $k$ can be slow is by having several sub-boxes at the previous scale $k-1$ which are slow or at least one which is blocked. This is made precise in the following Lemma.

Recall the definitions of $J_m$ in \eqref{e:Jm} and $M_m^j$ below
\eqref{e:supT}.
\begin{lemma}
\label{l:cure}
Fix $m \in M_k$ with $k \geq 1$, and suppose that $W_m$
holds. Then either
\begin{align}
 \label{e:Zslow}
 & \text{more than $3 L_{k-1}^a D$ indices $j \in J_m$ contain a slow box ($C_{m'}$ with $m' \in M_m^j$),}\\
 \label{e:oneblocked}
 & \text{or there exists at least one $m' \in M_k$ such that $C_{m'} \subset C_m$ and $T_{m'} = \infty$.}
\end{align}
\end{lemma}

\begin{proof}
Suppose that \eqref{e:oneblocked} does not hold. This implies by \eqref{e:Tmfinite} that
\begin{equation}
  \sup_{m' \in M_m^j} T_{m'}(\omega_{m'}) \leq
  (3 h)^{d-1}  L_{k-1}^{d+a}, \text{ for every $j \in J_m$.}
\end{equation}

Let us also assume that \eqref{e:Zslow} does not hold. This implies that
\begin{equation}
  \label{e:twoslow2}
  \begin{array}{c}
  \text{for all but $3 L_{k-1}^a D$ indices $j \in J_m$ we have}\\
  \sup_{m' \in M_m^j} T_{m'}(\omega_{m'}) \leq r_{k-1}L_{k-1}.
  \end{array}
\end{equation}
Joining the two above statements with Lemma~\ref{l:supT}, we can estimate
$T_m$ as follows
\begin{equation}
  \begin{split}
  T_m  \leq \smash{\sum_{j \in J_m} \sup_{m' \in M_m^{j}}}
  T_{m'}(\omega_{m'})
  & \leq |J_m| r_{k-1}L_{k-1} + 3 L_{k-1}^a D (3 h)^{d-1}  L_{k-1}^{d+a} \\
  & \overset{\mathclap{\smash{h > 1}}}\leq L_k \Big( r_{k-1} + D (3h)^d \; \tfrac{L_{k-1}^{d+2a}}{L_k} \Big)\\
  & \overset{\mathclap{\smash{\eqref{e:rk}}}}= r_k L_k.
  \end{split}
\end{equation}
This finishes the proof of Lemma~\ref{l:cure} by contraposition.
\end{proof}

The main idea behind the multi-scale renormalization procedure is to bound
the quantity $w_k$ defined in \eqref{e:slow} by $w_{k-1}$. This is
done with the help of the following

\begin{proposition}
\label{p:recur}
If we define the scale lengths $(L_k)_{k \geq 1}$ as in \eqref{e:Lk} with $L_0 \geq 100 \vee \useconstant{c:cov}$, then for any given $k \geq 1$ we have
\begin{equation}
  \label{e:skrecur}
  w_k
  \leq (3 L_{k-1}^{\gamma-1})^{d D} \Big[ w_{k-1}^D + L_{k-1}^{-\alpha (a+1)} \Big]
  + (3 L_{k-1}^{\gamma-1})^d v_{k-1} \, .
\end{equation}
\end{proposition}

\begin{proof}
Using Lemma~\ref{l:cure}, we obtain that for any $k \geq 1$ and $m \in M_k$,
\begin{equation}
  \label{e:twocases}
  \begin{split}
  \mathbb{P}[W_m] & \leq \mathbb{P} \big[ \text{more than $3 L_{k-1}^a D$ indices $j \in J_m$ contain a slow $C_{m'} \subseteq C_m$}\big]\\
  & \quad + \mathbb{P} \big[ V_{m'} \text{ occurs for some $m' \in M_{k-1}$ with $C_{m'} \subseteq C_m$ and index $j \in J_m$} \big] \, ,
  \end{split}
\end{equation}
Let us recall that the number of $m' \in M_{k-1}$ such that $C_{m'} \subseteq C_m$ and with index $j \in J_m$ (see \eqref{e:Jm}) is bounded by $(3  L_{k-1}^{\gamma-1})^d$. Therefore, the probability of the event on the second line of \eqref{e:twocases} is bounded by $(3 L_{k-1}^{\gamma-1})^d v_{k-1}$, which corresponds to the second term in \eqref{e:skrecur}.

Let us now turn to the probability appearing in the first line of \eqref{e:twocases}. We first order the indices $j_1 < j_2 < \dots < j_{3L^a_{k-1}D + 1}$ containing a slow $C_{m'} \subseteq C_m$. Then, by considering the sub-sequence $j'_i = j_{i 3L^a_{k-1}D + 1}$, for $i = 1, \dots, D$, we obtain
\begin{align}
 \nonumber
  \big\{ \text{mo}&\text{re} \text{ than $3 L_{k-1}^a  D$ indices $j \in J_m$ contain a slow box $C_{m'} \subseteq C_m$}\big\}\\
  &\subseteq \Big\{
  \begin{array}{c}
  \text{there are indices $j'_1, \dots, j'_D$ containing slow boxes $C_{m'_1}, \dots, C_{m'_D} \subseteq C_m$}\\
  \text{and such that $|j'_i - j'_{i'}| \geq 3 L_{k-1}^a$, for $i \neq i' \in \{1, \dots, D\}$}
  \end{array}\Big\}\\
 \nonumber
  &\subseteq \Big\{
  \begin{array}{c}
  \text{there are slow boxes $C_{m'_1}, \dots, C_{m'_D}$ ($m'_1, \dots, m'_D \in M_{k-1}$) contained}\\
  \text{in $C_m$ and such that $d(C_{m'_i}, C_{m'_{i'}}) \geq   L_{k-1}^{a+1}$, for $i \neq i' \in \{1, \dots, D\}$}
  \end{array}\Big\}.
\end{align}
Fixed boxes $C_{m'_1}, \dots, C_{m'_D}$ satisfying $d(C_{m'_i}, C_{m'_{i'}}) \geq L_{k-1}^{a+1}$, we can bound the probability that they are all slow by $w_{k-1}^D + L_{k-1}^{-\alpha (a+1)}$ using assumption \eqref{e:cov2} and the fact that $L_{k-1} \geq L_0 \geq \useconstant{c:cov}$. We can finally obtain \eqref{e:skrecur} by plugging this estimate with the number of possible choices for the $D$ boxes $C_{m'_1}, \dots, C_{m'_D}$ contained in $C_m$. This finishes the proof of Proposition~\ref{p:recur}.
\end{proof}

Proposition~\ref{p:recur} above provides us a way to bound  $w_k$, thanks to the bound \eqref{e:vkdecay} on $v_k$ which is a consequence of assumption  \eqref{e:W}.

For the following theorem, we suppose that  $\gamma  \geq d + 2a$, $D > 2 \gamma$ and that
\begin{align}
 \label{e:assumpt1}
  \alpha (a+1) & > 2 D d (\gamma-1) \quad \text{ and}\\
 \label{e:assumpt2}
  \rho & >  d (\gamma-1)(1+ 2 \gamma)
\end{align}

\begin{remark}
\label{r:constants}
It should be observed that we have plenty of possible choices for $D, \gamma, \alpha$ and $\rho$ satisfying \eqref{e:assumpt1}--\eqref{e:assumpt2} in order to apply Theorem~\ref{t:sdkdecay} below. These different choices will lead to complementary results. Roughly speaking, we should keep in mind that:
\begin{align*}
 & \text{the smaller the $\alpha$, the more dependent can be the environment $\omega$ (see \eqref{e:cov2}),}\\
 & \text{the larger $d (\gamma -1)$, the faster will be the decay of $w_k$ (see \eqref{e:skdecay}) and}\\
 & \text{one may only be able to prove \eqref{e:vkdecay} for certain small values of $\rho$.}
\end{align*}
In particular, if the disorder correlations and the probability in \eqref{e:W} are both exponentially decaying then \eqref{e:assumpt1}--\eqref{e:assumpt2} can always be verified, by properly choosing $\gamma$, $D$, $\alpha$ and $\rho$ in this precise order.
\end{remark}

\begin{theorem}
\label{t:sdkdecay}
Assume that \eqref{e:cov2} and \eqref{e:W} hold, with $D$, $\alpha$ and $\rho$ as in \eqref{e:assumpt1}--\eqref{e:assumpt2}.
Then, we can choose $L_0$ and $r_0$ large enough (depending on $\gamma$, $d$, $D$, $\alpha$, $h$, $\useconstant{c:cov}$ and $\useconstant{c:unpinning}$), so that, defining the scale sequence $(L_k)_{k \geq 0}$ as in \eqref{e:Lk}, we have
\begin{equation}
  \label{e:skdecay}
  w_k \leq L_k^{- 2 d (\gamma -1)}.
\end{equation}
\end{theorem}

\begin{proof}
Let us first choose $L_0 \geq \useconstant{c:cov} \vee \useconstant{c:unpinning}$ (see \eqref{e:cov2} and \eqref{e:W}) large enough such that
\begin{equation}
 \label{e:L0bound}
 \inf \left\{
 L_0^{(\gamma - 1)d (D - 2\gamma )} ,  L_0^{ - d (\gamma-1) (1+2 \gamma) + \rho}
 \right\}
 \geq 4 \cdot 
 3^{Dd},
\end{equation}
which is possible due to $D > 2 \gamma$ and \eqref{e:assumpt2}. Note that $L_0$ depends on $\gamma$, $d$, $D$, $\alpha$, $h$, $\useconstant{c:cov}$ and $\useconstant{c:unpinning}$.
This inequality will be used later with $L_k$ in place of $L_0$ as the sequence $\{ L_k\}_k$ is increasing.


Then we pick $r_0(h, \useconstant{c:cov}, \useconstant{c:unpinning})$ large enough so that \eqref{e:skdecay} holds with $k = 0$ (this determines the remaining values of $(r_k)_{k\geq 1}$ through \eqref{e:rk}).
All we need to show is that, if for some $k \geq 1$,
\begin{equation}
  \label{e:induct}
  w_{k-1} \leq L_{k-1}^{- 2 d (\gamma -1) },
\end{equation}
then $w_{k} \leq L_{k}^{- 2 d (\gamma -1) }$.

This is done with help of Proposition~\ref{p:recur} together with \eqref{e:induct} and \eqref{e:vkdecay}, leading to
\begin{equation}
  \begin{array}{e}
   w_k & \leq & (3L_{k-1}^{\gamma-1})^{dD}\Big[w_{k-1}^D + L_{k-1}^{-\alpha (a+1)} \Big] + (3L_{k-1}^{\gamma-1})^dL_{k-1}^{-\rho}\\[1mm]
       & \overset{\eqref{e:assumpt1}}\leq & 2 (3)^{dD} L_{k-1}^{(\gamma-1)d D - 2 d D (\gamma -1) }
       + (3)^d L_{k-1}^{(\gamma-1)d  - \rho }\\[1mm]
       & \overset{\eqref{e:assumpt2}}\leq &
       2 (3)^{dD} L_{k-1}^{(\gamma-1)d (2 \gamma - D)}  \; \big( L_{k-1}^\gamma \big)^{- 2 d(\gamma-1)}
       + 3^d \; L_{k-1}^{d (\gamma-1) (1+2 \gamma) - \rho}
       \; \big( L_{k-1}^\gamma \big)^{-  2 d (\gamma-1)}\\
       & \overset{
       \eqref{e:L0bound}}\leq &
       L_k^{-  2 d (\gamma-1)} \, .
  \end{array}
\end{equation}
This concludes by induction the proof of Theorem~\ref{t:sdkdecay}.
\end{proof}

Theorem~\ref{c:speed} can now be deduced as  a simple consequence of Theorem~\ref{t:sdkdecay}.

\begin{proof}
We choose $L_0$ and $r_0$ as in Theorem~\ref{t:sdkdecay} and define the sequences $(L_k)_{k \geq 1}$ as in \eqref{e:Lk}.
For any given $k \geq 1$, let $m_k = (k, (0,0))$ and define
$$
M_k' = \{m' \in M_{k-1}; B_{m'} \subset C_{m_k}, \text{with indices in $J_{m_k}$} \}
$$
where $M_k$ was introduced in \eqref{e:Mk}.
Let us estimate,
\begin{equation*}
\sum_{k \geq 1} \sum_{m' \in M_k'} \mathbb{P}[W_{m'}] \leq \sum_{k \geq 1} (3h)^{d-1} L_{k-1}^{(\gamma-1)d} w_{k-1} \overset{\text{Theorem~\ref{t:sdkdecay}}}\leq \sum_{k \geq 1} (3h)^{d-1} L_{k-1}^{- (\gamma-1)d},
\end{equation*}
which is clearly finite (recall \eqref{e:assumpt1} and that $L_k$ grows faster than exponentially). Therefore, using Borel-Cantelli's lemma, we can conclude that $\mathbb{P}$-a.s.
\begin{equation}
\label{e:allTm}
 T_{m'} \leq \useconstant{c:boundrk} L_{k-1}, \text{ for every $m' \in M_k'$ and all but finitely many $k$'s,}
\end{equation}
where $\useconstant{c:boundrk}(h,\useconstant{c:cov}, \useconstant{c:unpinning})$ was defined in \eqref{e:boundrk}.

To finish the proof, we need to show that on the event \eqref{e:allTm}, we have \eqref{e:speed}. We first show that under \eqref{e:allTm}, we have
\begin{equation}
 \label{e:steps}
 S_{\useconstant{c:boundrk} n L_k}(0) \geq n L_k \text{ for every $n = 1, \dots,  L_{k}^{\gamma-1}$ and all but finitely many $k$'s}.
\end{equation}
To see why this is true, fix some $k \geq 1$ for which \eqref{e:allTm} holds. Then we can use \eqref{e:stairs} to conclude that after time $n \useconstant{c:boundrk} L_{k-1}$, $S(0)$ has surpassed $n L_{k-1}$, for every $n \leq  L_{k-1}^{\gamma-1}$.

To finish the proof of Theorem~\ref{c:speed}, let observe that the event in \eqref{e:steps} implies that $\liminf_t \frac{S_t(0)}{t} \geq 1/(2h\useconstant{c:boundrk})$.
To see why this is true, observe that for any $k$ such that \eqref{e:speed} holds, by the monotonicity of $S_t(0)$,
\begin{equation}
 S_t(0) \geq t/(2h\useconstant{c:boundrk}), \text{ for each integer $t \in [\useconstant{c:boundrk} L_{k-1}, \useconstant{c:boundrk} L_k]$.}
\end{equation}
This finishes the proof that $\liminf_t \frac{S_t(0)}{t} \geq 1/(2h\useconstant{c:boundrk})$ on the almost sure event \eqref{e:steps}, yielding  Theorem~\ref{c:speed}.
\end{proof}

\section{Application of the criterion}
\label{s:perturb}


In this section we consider several types of models and check that  criterion \eqref{e:W} is valid
when the mean forcing is strong enough. This implies by Corollary~\ref{c:speed} that the velocity of the interface is positive.
Furthermore for a specific type of evolution of Lipschitz surfaces, we check that criterion \eqref{e:W} is valid up to the pinning threshold.

\subsection{Strong interaction}

We will consider a model of Lipschitz interface evolution to illustrate the techniques of the previous sections.
Although several models could be constructed following the same lines as below, we will focus here on a specific example of $2$-Lipschitz interfaces with evolution rule that can be informally described as follows.
\begin{enumerate}
 \item[(i)] Whenever an update could lead to the surface having discrete gradient $\pm 3$ at some position, this update is suppressed.
 \item[(ii)] Respecting the above, whenever an update reduces the absolute value of the gradient from two to one at some position, it is performed.
 \item[(iii)] If the update cannot be determined by the above rules, it will occur if and only if the discrete Laplacian is greater than the depth of the potential.
\end{enumerate}
To make the above description more precise, we define the evolution according to \eqref{e:evolve}, with
\begin{equation}
 \label{e:FLips}
 \begin{split}
   F(a_1, & a_2, \dots, a_{2 (d-1) }, \omega)\\
   & = 1_{\{a_k > -2; \text{ for all
     } k \leq 2 (d-1)  \}} \; \Max \Big( 1_{\{a_k = 2; \text{ for some } k
     \leq 2 (d-1)  \}};
   1_{\{\Delta S(x) + \omega(x,S(x)) > 0\}} \Big).
 \end{split}
\end{equation}
We consider an environment where the disorder is made of independent Bernoulli variables taking values tuned with respect to the dimension $d$
$$
p = \mathbb{P} \big(  \omega(x,y)  = - 3 (d-1) \big) = 1 - \mathbb{P} \big( \omega(x,y)  = 1/2 \big) \, .
$$
If the absolute value of gradients at one site are all less or equal to 1, then the interface is blocked at this site when  $\omega  = - 3 (d-1)$ and will move when $\omega  = 1/2$ only if  the Laplacian at this site is non negative.

\begin{figure}[ht]
\centerline{
    \includegraphics[angle=0, height=0.15\textwidth]{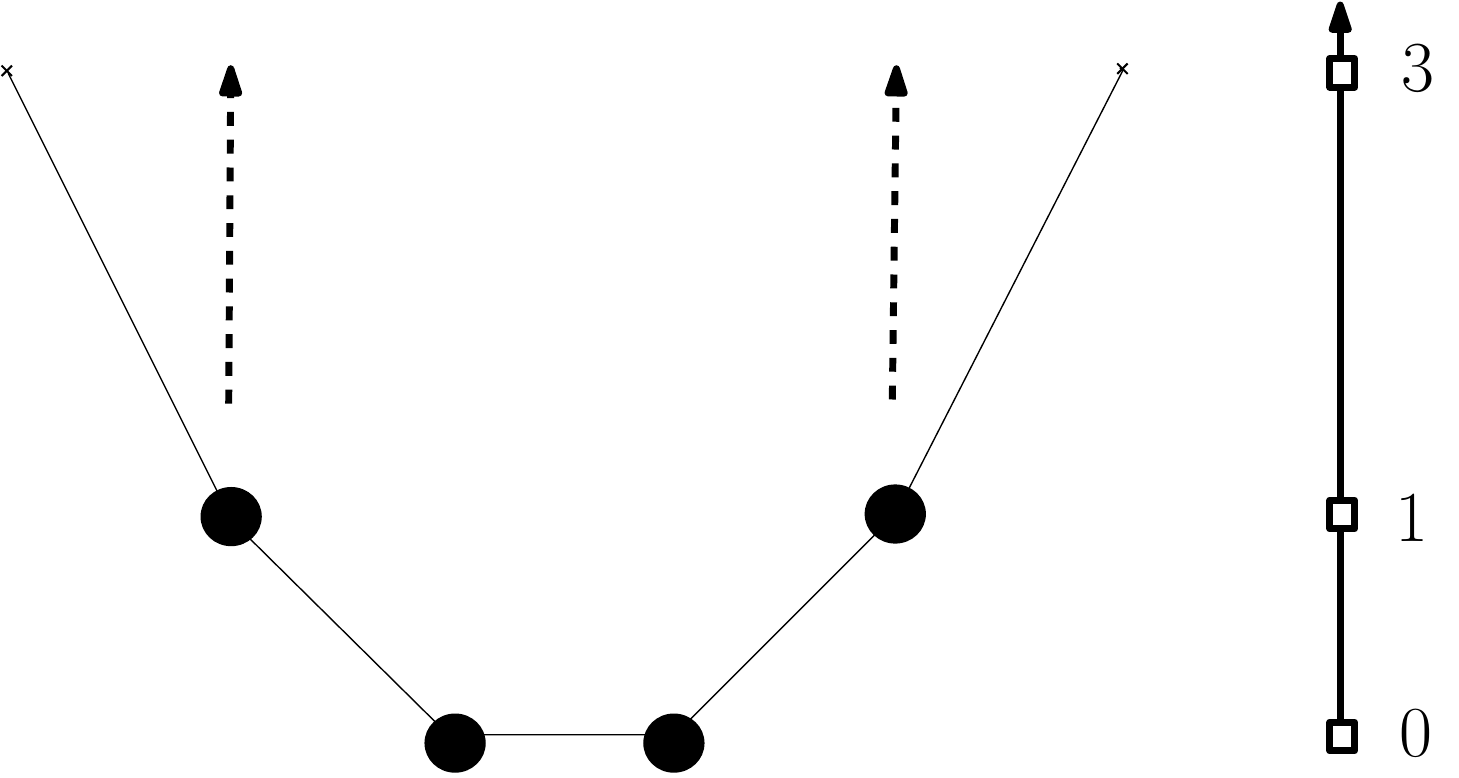}
}
    \caption{\small For $d=2$,  a portion of interface is depicted and the dots mark the occurrence of a trap $\omega = -3$.
    The heights are represented on the right.
    For this particular trap configuration, we note that the interface would be always blocked if condition (ii) was not imposed. Indeed, the outermost sites are blocked because of the $2$-Lipschitz constraint and
the 4 sites inside satisfy $\Delta S(x) + \omega(x,S(x)) \leq 0$. Condition (ii) implies that the interface cannot be trapped by a local configuration :  for the interface depicted above, the two sites indicated by the arrows will first move upward leading, in the next step, to unblocking the two sites at the center.}
    \label{fig: blocage}
\end{figure}

The set of $2$-Lipschitz interfaces is invariant under the previous evolution. Note that an evolution encoded only by a Laplacian and restricted to the class of $2$-Lipschitz interfaces would display a trivial behavior for all values of $p>0$ as the interfaces get almost surely blocked by single local trap (see figure \ref{fig: blocage}).
The rule (ii) of forcing a move when a gradient is equal to 2 prevents the local traps from blocking the interface forever. We refer to \cite{KoltonRosso} where other rules have been devised  to suppress the effect of local traps in Lipschitz dynamics.  Similar evolutions have been also studied in \cite{BS95, TL92}.

\medskip

We first state
\begin{lemma}
 \label{l:blocklips}
 For $d \geq 2$ and $p$ close enough to one, there are almost surely infinitely many blocked surfaces.
\end{lemma}

\begin{proof}
If $S$ is a $1$-Lispschitz surface such that any site is pinned, i.e. $\omega(x,S(x))  = - 3 (d-1)$, then the evolution rules \eqref{e:FLips} imply that the interface is blocked.
For large enough $p$, the occurrence of percolating $1$-Lipschitz surfaces is a consequence of  Theorem 1 in \cite{DDGHS}. Thus the Lemma holds.
\end{proof}

We now turn to small values of $p$ proving the following
\begin{lemma}
 \label{l:unblocklips}
For $d \geq 2$ and $p > 0$ small enough, then for $h \geq 10$ the criterion \eqref{e:W} holds. Consequently, the speed of the surface is positive.
\end{lemma}

\begin{proof}
 Given $h \geq 10$ and $a=1$, consider a surface $S$ blocking $C^{h,L}$ as below \eqref{e:CB}. We will first prove that there is $x_0$ in  $[hL, 2hL]^{d-1}$ such that
\begin{align}
 \label{e:Sbelow}
  & \text{the blocked surface $S$ restricted to $\cD^L(x_0) = x_0 + [-L, L]^{d-1}$ is below $3L$}\\
 \label{e:1lips}
  & \text{and $S$ is $1$-Lipschitz in $\cD^L(x_0)$.}
\end{align}
The first claim follows from the fact that the blocking surface $S$ touches $B^{h,L}$ so that there is a site $x_0$ in  $[hL, 2hL]^{d-1}$ where $S(x_0) \le L$.
Since  $S$ is $2$-Lipschitz, the surface remains below $3 L$ in $\cD^L(x_0)$.
To prove the second claim, suppose that there exists $x \in \cD^L(x_0)$ and a direction $e \in E$ such that $S(x+e) - S(x) = -2$ (the case $S(x+e) - S(x) = 2$ is analogous). Let $\{x_k\}_{k \leq \bar n}$ be the longest sequence of sites such that $x_0 =x$,  $x_1 =x +e$, $|x_k - x_{k-1}| = 1$ and $S(x_k) - S(x_{k-1}) = -2$.
It is clear that $\bar n \leq hL$, since $S(x)$ is positive in $[0, 3hL]^{d-1}$. But then, according to \eqref{e:FLips}, the surface is not blocked in $x_{\bar n}$, contradicting \eqref{e:Sbelow} and proving \eqref{e:1lips}.

The last property of such a blocking surface we need is that there exists a $c > 0$ depending only on the dimension $d$, such that
\begin{equation}
 \label{e:bend}
  \arraypar{for every box $B$ of side length $c$ contained in $\cD^L(x_0)$, we have $\omega(x, S(x)) = - 3(d-1)$ for some $x \in B$.}
\end{equation}
To see why this is true, observe from \eqref{e:Sbelow} and \eqref{e:1lips}, that given any point $x \in [hL, 2hL]^{d-1}$ we either have $\omega(x, S(x)) = -3(d-1)$ or $\Delta S(x) \leq -1$, see \eqref{e:FLips}. We will need the following discrete version of the Divergence Theorem
\begin{equation}
 \label{e:discreteGauss}
 \sum_{x \in B} \Delta S(x) = \sum_{x \in B, y \in \Z^{d-1} \setminus B; \atop |x-y|=1} S(y) - S(x),
 \quad  \text{ for every box $B \subset \Z^d$}.
\end{equation}
With this, \eqref{e:bend} follows from the fact that the ratio between the size of the boundary and the volume of a box goes to zero, so that for large enough boxes it is not possible to have a negative Laplacian on every point of the box and keep its Lipschitz character. This proves \eqref{e:bend}.

From \eqref{e:bend}, it is easy to see that given $x_0$ in $[hL, 2hL]^{d-1}$ and a  fixed surface $S: \cD^L(x_0) \to [0, 3L]$
\begin{equation}
 \label{e:expdecayS}
 \arraypar{the probability that $S$ is blocked is smaller or equal to $\exp \{-\psi(p) L^{d-1}\}$,}
\end{equation}
where $\psi(p)$ converges to infinity as $p$ goes to zero. Indeed, if one defines a paving of $[hL, 2hL]^{d-1}$ with boxes of side length $c$, each of these boxes must have at least one trap according to \eqref{e:bend}.

We complete the proof of the Lemma with a counting argument:
\begin{equation}
  \begin{split}
   \mathbb{P}[\text{there exists $S$ blocking $C^{h,L}$}] & \leq \sum_{x_0 \in [hL, 2hL]^{d-1}} \, \sum_{S: \cD^L(x_0) \to \{0,\dots, 3L\} \atop \text{$1$-Lipschitz}} \mathbb{P}[\text{$S$ is blocked}]\\
   & \leq \sum_{x_0 \in [hL, 2hL]^{d-1}}\,  \sum_{S:\cD^L(x_0) \to \{0,\dots, 3L\} \atop \text{$1$-Lipschitz}} \exp \{ -\psi(p) L^{d-1} \}.
  \end{split}
 \end{equation}
But since the number of such surfaces is bounded by $hL^2 \,  3^{d L^{d-1}}$, the above probability decays exponentially in $L$ as long as $p$ is sufficiently small, proving Lemma~\ref{l:unblocklips}. To see why the above implies the positive speed of the surface dynamics, one can simply use Remark~\ref{r:constants} to find constants satisfying \eqref{e:assumpt1}--\eqref{e:assumpt2} and then apply Theorems~\ref{t:sdkdecay} and \ref{c:speed}.
\end{proof}

In dimension $d =2$, the criterion  \eqref{e:W} is sharp for these evolution rules
\begin{lemma}
\label{l:unblockd=2}
For $d = 2$, there exists a critical parameter $p_c \in ]0,1[$ such that
\begin{align}
  \label{e:sub_crit}
  & \text{if $p< p_c$, \eqref{e:W} holds and the surface moves $\mathbb{P}$-a.s. with positive speed,}\\
  \label{e:super_crit}
  & \text{if $p>p_c$, the interface gets blocked almost surely, i.e. $\lim_t S_t(x) < \infty$, for each $x$.}
\end{align}
\end{lemma}
\begin{proof}
First, we will explain a correspondence between blocked interfaces (for the evolution rules \eqref{e:FLips}) and a directed percolation model.
 A site $(x,y)$ is called a {\it  blocking site} if $\omega (x,y) = -3$.
As explained in \eqref{e:1lips}, a blocked interface is $1$-Lipschitz in an interval of length at least $2L$.

Let $S$ be a $1$-Lipschitz blocked interface and $(x,S(x))$ be a blocked site.
Then, the next blocking site on the interface  to the right of $x$ can take at most 7 locations which are depicted figure \ref{fig: directed percolation}.
If none of these 7 sites are blocked then the interface cannot be blocked.
Thus a $1$-Lipschitz blocked interface is in correspondence with an oriented percolation path in the graph $\mathbb{L}$, whose vertices are $\mathbb{Z}^2$ and the oriented edges are given by
\begin{equation}
  \label{e:edges_E}
  \mathcal{E} = \Big\{ \big((x,y),(x + i, y + j)\big); x, y \in \mathbb{Z}^d \text{ and } (i, j) \in E \Big\},
\end{equation}
where $E = \{(1,1), (2,1), (3,0), (2,0), (1,0), (2,-1), (1,-1)\}$. See Figure~\ref{fig: directed percolation} for an illustration of all edges departing from a given vertex $(x,y)$.

\begin{figure}[ht]
\centerline{
    \includegraphics[angle=0, height=0.15\textwidth]{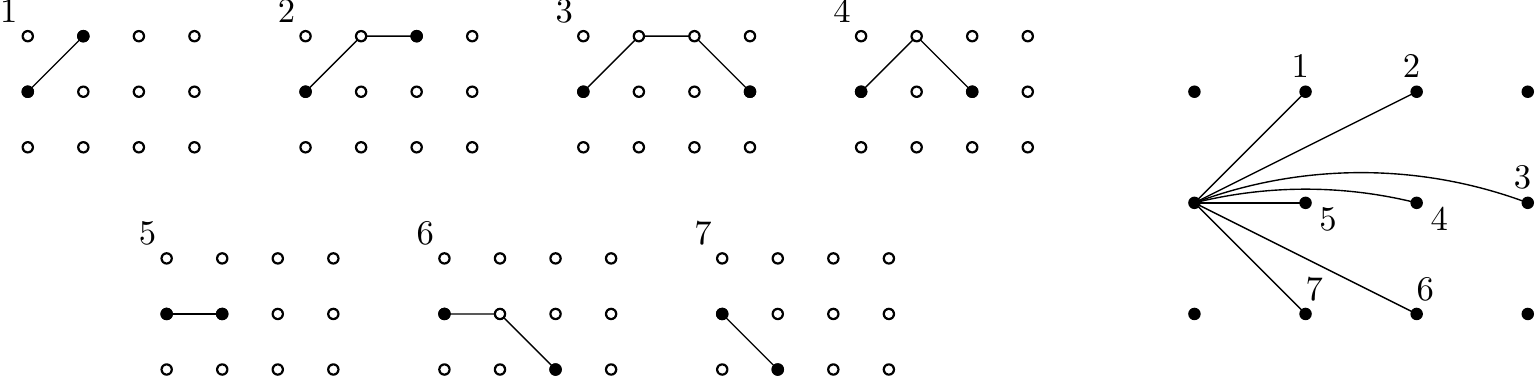}
}
    \caption{\small The left picture illustrates all possible situations of blocked sites which are represented by black dots. On the right,  their corresponding edges are depicted and they match the set $\mathcal{E}$.}
    \label{fig: directed percolation}
\end{figure}

Let $p_c$ be the critical threshold for  this oriented percolation on $\mathbb{L}$.
By Aizenmann-Barsky-Menshikov's Theorem (see Theorem~9 p.106 of \cite{MR2283880}), the probability of finding a percolating path to distance $L$ from the origin decays as $\exp ( -c L)$ when $p<p_c$.
Any blocked interface touching $B^{h,L}$ will be $1$-Lipschitz
in a stretch of length at least $2L$  thanks to \eqref{e:1lips}, thus the criterion  \eqref{e:W} holds  for $p< p_c$ with an exponential decay instead of a polynomial one.
Using Theorem \ref{c:speed} we now prove \eqref{e:sub_crit}.

\medskip

Suppose now that $p>p_c$ and let us show that
\begin{equation}
  \label{e:blockabove}
  \mathbb{P}[\text{there exists a blocked interface $S \geq 0$}] = 1.
\end{equation}
It is clear that the above probability is either $0$ or $1$ by ergodicity arguments.
We now show that it is positive, but first let us provide some further simplifications.

It is enough to show that
\begin{equation}
  \label{e:one_ray}
  \mathbb{P}\big[\text{there is an infinite occupied path $\overset{\rightarrow}\gamma$ in $(\mathbb{Z}_+\times\mathbb{Z}_+, \mathcal{E})$ starting from $0$} \big] > 0.
\end{equation}
Indeed, let us suppose that the above is true and conclude \eqref{e:blockabove}.
Note first that our percolation model on $\mathbb{L}$ is symmetric with respect to reflections around the $y$-axis (once we invert the orientation of all edges). This symmetry, together with \eqref{e:one_ray}, says that with positive probability there exists an infinite occupied path $\overset{\leftarrow}\gamma$ on $\mathbb{Z}_- \times \mathbb{Z}_+$ starting from the origin and crossing the edges of $\mathbb{L}$ in the reversed direction.
By concatenating two such paths and using the i.i.d. nature of our percolation model, we obtain that with positive probability there exists a doubly infinite occupied path $\overset{\leftrightarrow} \gamma: \mathbb{Z} \mapsto \mathbb{Z} \times \mathbb{Z}_+$, which proves \eqref{e:blockabove}.

Before turning to the proof of \eqref{e:one_ray}, let us first recall that similar statement holds true in a more standard context. Consider the classical oriented percolation on the graph $(\mathbb{Z}_+ \times \mathbb{Z}, \mathcal{E}')$, where $\mathcal{E}'$ is defined as in \eqref{e:edges_E} with $E$ replaced by $E' = \{(1,1), (1,-1)\}$. The results in \cite{MR1048929} imply that in the super-critical percolation regime, the claim $\eqref{e:one_ray}$ holds.

From the above reasoning, to prove \eqref{e:one_ray} it is enough to construct a process of renormalized boxes on $\mathbb{Z}_+ \times \mathbb{Z}$ which is dominated by our super-critical percolation, but on the other hand dominates a standard oriented percolation with arbitrarily high parameter $\eta$.
To achieve this, we follow \cite{MR1071804} and use a one step renormalization.
More precisely, we only have to show that for any $\eta > 0$ there exist $L(\eta), S(\eta), k(\eta) \geq 1$ and boxes $v_{i,j} = [0, 2S] \times [-2L, 2L] + (2ikS, jkL)$, for $i,j \in \mathbb{Z}_+ \times \mathbb{Z}$, such that
\begin{equation}
  \begin{array}{c}
    \text{the collection of $(i,j)$'s such that the box $v_{i,j}$ can be reached}\\
    \text{by an occupied path in $\mathbb{L}$ starting from $v_{0,0}$ dominates a standard}\\
    \text{oriented percolation from the origin with parameter $\eta$.}
  \end{array}
\end{equation}

The above claim follows from the arguments in \cite{MR1071804} (see Lemma~(21)) with minor modifications
(some of which are explained in \cite{hiemer}).
This completes the proof of \eqref{e:one_ray} and yields Lemma~\ref{l:unblockd=2}.
\end{proof}

\begin{remark}
\label{r:dependent}
\

\begin{enumerate}
\item Consider for instance the dynamics induced by \eqref{e:FLips} for $d = 2$. The proof of Lemma~\ref{l:blocklips} could easily be extended for non independent environments. Indeed, the only part of the proof where we used the i.i.d. structure of the energy landscape was to prove the decay of $\mathbb{P}[\text{there exists $S$ blocking $C^{h,L}$}]$. This decay could be established for dependent models such as finite dependent environments (see for instance Theorem~(7.61) of \cite{Gri99}, p.178). Environments with longer range of dependence (say, satisfying \eqref{e:cov} for a large enough $\alpha$) can probably be dealt with using renormalization techniques such as the ones appearing in Section~4 of \cite{Szn09} (see also Section~8 of \cite{PT12}).
\item We also believe that the techniques in \cite{MR1048929} could be used to show \eqref{e:one_ray}.
\end{enumerate}
\end{remark}

\subsection{Soft interaction}

In this section we are going to study non-Lipschitz surfaces in the presence of unbounded traps. In what follows, we restrict ourselves to $d = 2$.

We consider an evolution of the form
\begin{equation}
  \label{e:evolve2}
  S_{t+1}(x) = S_t(x) + F \Big(S_t(x-1) \negthinspace - \negthinspace S_t(x) \, , \; S_t(x+1) \negthinspace - \negthinspace S_t(x), \, \; \omega(x,S_t(x)) \Big) \, ,
\end{equation}
with $F(a,b,\omega) = 1_{b + a + \omega > 0}$. The intuitive description of the above mechanism is that the surface will move whenever its discrete Laplacian in a point overcomes the depth of the potential in that point.

The main result of this section is the following lemma which establishes the validity of the finite size criterion \eqref{e:W} for dynamics with evolution rule \eqref{e:evolve2} when the driving force is large enough (see assumption \eqref{e: perturbative condition 2} on the disorder).
In this framework, the positive velocity was already derived in a perturbative regime in \cite{DS12} and we follow their approach to check the finite size criterion under some conditions on the disorder distribution.

\begin{lemma}
\label{l:mainassump}
Suppose that the variables $\omega(x,y)$ are  i.i.d. in $C^{h,L}$ and satisfy
\begin{align}
&\exists \lambda_0 >0, \qquad
\mathbb{E}( \exp( - \lambda_0 \, \omega(x,y) ) < 1 - \exp( - \lambda_0) \, .
\label{e: perturbative condition 2}
\end{align}
Furthermore all the $\omega$ outside $C^{h,L}$ are equal to $- \infty$. Choose $a = 2$ such that the box $C^{h,L}$ has height $H_L = L^3$. Then there exist $c_1, \kappa > 0$ and $h > 1$, such that, for every $L \geq c_1$
\begin{equation}
\label{e:Assumption}
\mathbb{P}\big[ \text{there exists a surface $S$ blocking $C^{h,L}$} \big]
\leq \exp( - \kappa \sqrt{L} ) \, .
\end{equation}
\end{lemma}
Assumption \eqref{e: perturbative condition 2} controls the statistics of the large negative forces which may pin down the interface. Note that the decay in \eqref{e:Assumption} is not optimal.

\medskip

\begin{remark}
If $\omega(x,y)$ are independent Gaussian variables with mean $f >0$ and variance $\sigma$,
then condition \eqref{e: perturbative condition 2} reads
\begin{eqnarray*}
\exp \left( - \lambda_0  f + \frac{\sigma}{2} \lambda_0^2 \right) < 1 - \exp( - \lambda_0),
\end{eqnarray*}
which  holds  for an  appropriate choice of $\lambda_0$ when $f$ is large enough.
\end{remark}

\begin{proof}
The proof is split into 3 steps.

\medskip

\noindent
{\bf Step 1.}

First, we are going to control the upward fluctuations of the blocked interfaces.
We introduce
\begin{equation}
\label{eq: D}
\mathcal{D} := \{\omega(x,y) \geq - \sqrt{L}, \text{ for every $(x,y) \in C^{h,L}$}\}
\end{equation}
which has probability at least $1 - 3 h L^4 \exp( - \lambda_0 \sqrt{L})$. This follows from bounding the tail of a single variable
\begin{equation}
\label{eq: D borne}
\begin{split}
  \mathbb{P} \big( \mathcal{D}^c \big) = 3 h L^4 \mathbb{P} \big(
  \omega < - \sqrt{L} \big) & \leq 3 h L^4 \exp( - \lambda_0 \sqrt{L} )
  \; \mathbb{E} \big( \exp( - \lambda_0 \omega ) \big)\\
  & \leq 3 h L^4
  \exp( - \lambda_0 \sqrt{L} ) \, ,
\end{split}
\end{equation}
where we used \eqref{e: perturbative condition 2} to conclude.

Let $S$ be a blocked interface such that $S(x_0+1) -S(x_0) \leq 1$, then
on the event $\mathcal{D}$, the maximum principle implies that $S$  remains below the interface $\mathcal{H}$ defined by
\begin{equation*}
\begin{cases}
\mathcal{H} (x_0) = S(x_0), \quad
\mathcal{H} (x_0+1) = S(x_0) +1  \\
\mathcal{H}(x+1) = - \mathcal{H}(x-1) +  2 \mathcal{H} (x)  + \sqrt{L}, \qquad x \geq x_0+1
\end{cases}
\end{equation*}
The solution is given by
\begin{equation}
\label{eq: H comparaison}
\mathcal{H}_{x_0,S(x_0)} (x) = \frac{\sqrt{L}}{2} \;  (x-x_0) (x-x_0 -1) + S(x_0) + x - x_0 \, ,
\end{equation}
where we added the subscript to stress the initial data.

If $S(x_0) \leq L$ then $\mathcal{H}_{x_0,S(x_0)}$ does not reach the top layer of $C^{h,L}$, so that  $S$ will never reach the height $H_L = L^3$.

\medskip

\noindent
{\bf Step 2.}


Let us introduce some  notation.
For a surface $S$, we denote by $V(x) = S(x+1) - S(x)$ its increment at $x$.
For $\ell \leq n \leq 3 hL$, define $\mathcal{O}_{\ell,n}$ the set of surfaces in $C^{h,L}$ which are obstructed for all $\ell < x \leq  n$ and with a controlled height and a controlled initial gradient at $\ell$ that is, such that
\begin{align}
\label{e: obstructed gauss 1}
\forall x \in [\ell +1, n], \qquad
& V (x) \leq   V (x-1) \negthinspace - \negthinspace   \omega(x,S (x)) \, , \\
\label{e: obstructed gauss 2}
& S(\ell) \leq 2 L , \quad
V(\ell) \leq 1 \, ,\\
\label{e: obstructed gauss 3}
\forall x \in [\ell, n], \qquad & S(x) < \mathcal{H}_{\ell, 2 L} (x) \, .
\end{align}
Note that
\begin{equation}
 \label{e:SinO}
 \arraypar{the fact that a given surface $S$ belongs to $\mathcal{O}_{\ell,n}$ only depends on $S(\ell), \dots, S(n+1)$ and $\omega(\ell, S(\ell)), \dots, \omega(n+1, S(n+1))$.}
\end{equation}


\medskip

For $h>1$, we are going to check that
\begin{equation}
\label{eq: inclusion}
\mathcal{D} \cap \big\{ \exists \ \text{a surface $S$ blocking } C^{h,L} \big\}
\subseteq \bigcup_{x_0 \leq 2hL} \mathcal{O}_{x_0,3 hL} \, .
\end{equation}
First note that any surface $S$ blocking $C^{h,L}$ satisfies the constraint \eqref{e: obstructed gauss 1}
with $\ell = x_0$ and $n = 3hL$.
If an interface remains below the level $L$ in $\{0,2hL\}$ then \eqref{e: obstructed gauss 2}
has to be satisfied for at least one site $x_0$.

We now establish (\ref{e: obstructed gauss 2}).
If the interface $S$ is blocking $C^{h,L}$ but goes beyond the level $L$ before $2 hL$,  then define $z$ the first time at which $S(z)> L$ and $S(z+1) \leq L$. Either $S(z+2) - S(z+1) \leq 1$ and
\eqref{e: obstructed gauss 2} is satisfied for $x_0 = z+1$, or $S(z+2) - S(z+1) > 1$ and using the fact that the interface is obstructed  and the environment belongs to $\mathcal{D}$, one has
$$
S(z)  \leq S(z+1) - \big( S(z+2) - S(z+1) \big) - \omega(z+1,S(z+1))
\leq L + \sqrt{L} \, .
$$
Thus \eqref{e: obstructed gauss 2} is satisfied for $x_0 = z$.
Furthermore, we recall from step 1 that for environments in $\mathcal{D}$ the constraint \eqref{e: obstructed gauss 3} is implied by \eqref{e: obstructed gauss 1} and \eqref{e: obstructed gauss 2}.

\medskip

From \eqref{eq: inclusion} and \eqref{eq: D borne}, it is enough to bound from above the probability of each $\mathcal{O}_{x_0,3 hL}$.
The key estimate will be
\begin{lemma}
\label{lem: intermediaire}
There are $\kappa, \beta, \delta$ positive constants such that given any $x_0 \leq 2hL$,
\begin{equation}
 \label{eq: key upper bound}
 \begin{split}
 & \mathbb{P} \left[ \text{there exists $S  \in  \mathcal{ O}_{x_0,n}$ with } \; 
 \; V(n) \geq - \delta (n-x_0) \right] \\
 & \qquad \qquad \qquad \leq  \beta L \exp \big( - \kappa (n-x_0) \big), \qquad \text{ for every $n \in [x_0, 3hL]$.}
 \end{split}
\end{equation}
\end{lemma}
Loosely speaking Lemma \ref{lem: intermediaire} says that a typical surface in ${\mathcal{O}}_{\ell,n}$ has ``negative curvature'' in average with high probability. The proof is postponed to step 3.

\medskip

Inequality \eqref{eq: key upper bound} implies that
\begin{align}
\label{eq: key upper bound 2}
& \mathbb{P} \left[ \exists S  \in  \mathcal{O}_{x_0,3 hL}; \quad \exists k \in [\sqrt{L}, L], \quad
V(x_0+ k)  \geq - \delta k \right]\\
& \qquad \leq
\sum_{k = \sqrt{L}}^L
\mathbb{P} \left[   \exists S  \in  \mathcal{O}_{x_0, k};
\quad V(x_0+ k)  \geq - \delta k \right]
\leq \beta  L^2 \exp( - \kappa \sqrt{L}  ) \, . \nonumber
\end{align}
On the other hand, there does not exists a surface $S$ in $\mathcal{O}_{x_0,3 hL}$
such that
$$
\forall  k \in [\sqrt{L}, L], \qquad V(k +x_0) < - \delta k \, .
$$
Indeed, any such surface would satisfy for $L$ large enough
\begin{equation}
\label{eq: borne S}
S(L +x_0) - S( \sqrt{L} + x_0) \leq - \delta \sum_{k = \sqrt{L}}^L k \leq -  \frac{\delta}{4} L^2 \, .
\end{equation}
An interface $S$  in  $\mathcal{O}_{x_0,3hL}$ is below
$ \mathcal{H}_{x_0,2 L}$ \eqref{e: obstructed gauss 3} so that
$S( \sqrt{L} + x_0) \leq L^{3/2} + S(x_0)$. As $S(x_0) \leq 2 L$, this would imply that
$S(L+x_0) <0 $ and therefore lead to a contradiction.

Thus one deduces from \eqref{eq: key upper bound 2} that
\begin{equation}
\label{eq: key upper bound 3}
\mathbb{P} \Big[  \exists S  \in \textstyle \bigcup\limits_{x_0 \leq 2 hL} \mathcal{O}_{x_0,3 hL} \Big]
\leq 2h  \beta  L^3  \exp( - \kappa \sqrt{L} ) \, ,
\end{equation}
and combining this with \eqref{eq: D borne} completes the proof of Lemma~\ref{l:mainassump}.

\end{proof}

\medskip

\noindent
{\bf Step 3.}

In the last step, we complete the proof of Lemma~\ref{lem: intermediaire}.

\begin{proof}[Proof of Lemma~\ref{lem: intermediaire}]
%
%
%

Fix $x_0 \leq 2hL$ and $x_0 \leq n \leq 3hL$. Using \eqref{e: obstructed gauss 1} at site $n$, one gets
\begin{eqnarray*}
&& \mathbb{P} \left[  \exists S  \in  \mathcal{ O}_{x_0,n};
 \quad V(n) \geq - \delta (n-x_0) \right]\\
&& \quad
\leq
\mathbb{P} \Big[  \exists S_{x_0, n}  \in  \mathcal{O}_{x_0,n - 1};
S(n)< H_L, \quad   V(n-1) - \omega(n ,S(n))  \geq - \delta (n-x_0) \Big]
\end{eqnarray*}
where the  inequality is obtained by using only the constraint that the interface is obstructed  up to $n-1$.
The constraint $\mathcal{O}_{x_0,n-1}$ involves only the interface on the  sites in $[x_0,n]$, see \eqref{e:SinO}. Therefore, one can consider  portions of interfaces of the form $S_{x_0, n} = \{ S(i) \}_{x_0 \leq i \leq n}$ and bound
\begin{equation}
\begin{split}
& \mathbb{P}
 \bigg[
 \begin{array}{c}
  \exists S_{x_0, n}  \in  \mathcal{O}_{x_0,n - 1};
  S(n)< H_L, \quad   V(n-1) - \omega(n ,S(n))  \geq - \delta (n-x_0)
 \end{array}
 \bigg] \\
& \quad
\leq  \sum_{S_{x_0, n}} \mathbb{P}
 \bigg[
 \begin{array}{c}
  S_{x_0, n}  \in \mathcal{O}_{x_0,n-1}; S(n) < H_L,\\
  V(n-1) - \omega(n,S(n))  \geq - \delta (n - x_0)
 \end{array}
 \bigg] \\
& \quad
\leq \exp \big( \delta \lambda (n -x_0) \big)
\sum_{S_{x_0,n}}
\mathbb{E}
 \bigg[
 \begin{array}{c}
  \exp \Big( \lambda \big( V(n-1) - \omega(n,S(n) \big) \Big)\\
  1{\{  S_{x_0,  n}  \in \mathcal{O}_{x_0,n-1};  \ S(n)< H_L \}}
  \end{array}
 \bigg]
\end{split}
\end{equation}
We are now going to show that there exists $\beta >0$ such that for any $n \geq x_0$
\begin{equation}
\label{e: exp bound}
 \begin{split}
  \mathbb{A}_n^{x_0} & :=  \sum_{S_{x_0, n}} \mathbb{E}
 \bigg[
 \begin{array}{c}
  \exp \left( \lambda  V(n-1)  -  \lambda \omega(n,S(n)) \right)\\
  1{\{ S_{x_0, n} \in \mathcal{O}_{x_0,n-1}; \  S(n) < H_L \}}
 \end{array}
 \bigg]
 \leq
\beta  L \; \left( \frac{\mathbb{E}( \exp( - \lambda \omega) )}{1 - \exp( - \lambda) } \right)^{n-  x_0} \, .
 \end{split}
\end{equation}
Assumption \eqref{e: perturbative condition 2} implies for $\lambda = \lambda_0$ that $\mathbb{A}_n^{x_0}$  decays exponentially fast.
Thus, by choosing $\delta$ small enough, the Lemma will be complete.

\bigskip

It remains to check \eqref{e: exp bound}. An interface of form $S_{x_0, n}$ can be rewritten as the concatenation of $(S_{x_0,n-1},s)$, where $s$ is the height $S(n)$.
The constraint \eqref{e: obstructed gauss 1} leads to
\begin{equation*}
\begin{split}
\mathbb{A}_n^{x_0} & \leq
\sum_{(S_{x_0,n-1},s)} \mathbb{E}
 \bigg[
 \begin{array}{c}
  1{\{ (S_{x_0,n-1}, s)  \in \mathcal{O}_{x_0,n-1} ;   s < H_L \}}
  \exp \big( \lambda (s - S(n-1)) -  \lambda \omega(n,s) \big)
 \end{array}
 \bigg] \\
& \leq
\sum_{S_{x_0,n-1}} \mathbb{E} \Big[  1{\{ S_{x_0,n-1}  \in \mathcal{O}_{x_0,n-2}; S(n-1) < H_L \}}\\
& \qquad \qquad
\sum_{\mathclap{0 \leq s \leq S(n-1) \atop + V(n-2) - \omega(n-1,S(n-1) )}}
\exp \left( \lambda (s - S(n-1)) -  \lambda \omega(n, s) \right) \Big]
\end{split}
\end{equation*}
Conditioning on the first $n-1$ sites, one has
\begin{align}
\nonumber
\mathbb{A}_n^{x_0}  & \leq  \mathbb{E} \big[  \exp (- \lambda \omega) \big] \sum_{S_{x_0,n-1}} \mathbb{E} \Big[  1_{\{ S_{x_0,n-1}  \in \mathcal{O}_{x_0,n-2} ; \atop  S(n-1) \leq H_L  \}}
\sum_{s =0}^{\mathclap{S(n-1)  + V(n-2) \atop - \omega(n-1,S(n-1) )}}
\exp \left( \lambda (s - S(n-1))  \right) \Big] \\
& \leq
\frac{\mathbb{E} \big[  \exp (- \lambda \omega) \big]}{1 - \exp(- \lambda)}
\sum_{S_{x_0,n-1}} \mathbb{E}
 \bigg[
 \begin{array}{c}
  1{\{ S_{x_0,n-1}  \in \mathcal{O}_{x_0,n-2} ; \  S(n-1) < H_L  \}} \\
  \exp( \lambda  V(n-2) - \lambda  \omega(n-1,S (n-1)))
 \end{array}
 \bigg] \\
\nonumber
& \leq
\frac{\mathbb{E} \big[ \exp (- \lambda \omega) \big]}{1 - \exp( - \lambda)}
\mathbb{A}_{n-1}^{x_0}
\leq \dots \leq \beta L
\left(
\frac{\mathbb{E} \big[  \exp (- \lambda \omega) \big]}{1 - \exp( - \lambda)}
\right)^{n-x_0} \, ,
\end{align}
where we used the fact that $V(x_0) \leq 1$ for interfaces in $\mathcal{O}_{x_0,n}$
and that $S(x_0)$ takes at most $2 L$ values.
\end{proof}

\section{Conclusion and Open problems}

In this paper, we devised a renormalization procedure to study the surface motion in a random environment. When the finite size criterion \eqref{e:W} can be checked then our result implies that the interfaces move with positive velocity.
For a class of $2$-Lipschitz interface evolution, we have been able to prove that this criterion holds up to the pinning transition :  either the interfaces are blocked or they have a positive velocity.
Thus this rules out the possibility of an intermediate sub-ballistic regime for this type of dynamics.
The validity of criterion \eqref{e:W} for other classes of models is a challenging open question up to the pinning transition in particular for
unbounded gradient dynamics of the type \eqref{e:evolve2}.
In fact, it is still an open issue to prove that such dynamics have a positive velocity when the dimension $d$ is larger than 3 even for large driving forces.

The evolutions investigated in this paper are deterministic and the source of randomness is only due to the environment. Rephrased in physical terms, this means that we considered zero temperature dynamics. Many important physical problems  are related to interface motion
with positive temperature \cite{doussal, KoltonRosso} and it would be interesting to provide rigorous renormalization schemes in this new framework.

\bibliographystyle{plain}
\bibliography{all}

\begin{thebibliography}{10}

\bibitem{AizenmanBarsky}
Michael Aizenman and David~J. Barsky.
\newblock Sharpness of the phase transition in percolation models.
\newblock {\em Comm. Math. Phys.}, 108(3):489--526, 1987.

\bibitem{AizenmanChayesRusso}
Michael Aizenman, JT~Chayes, Lincoln Chayes, J~Fr{\"o}hlich, and L~Russo.
\newblock On a sharp transition from area law to perimeter law in a system of
  random surfaces.
\newblock {\em Communications in Mathematical Physics}, 92(1):19--69, 1983.

\bibitem{BS95}
Albert-L{\'a}szl{\'o} Barab{\'a}si and H.~Eugene Stanley.
\newblock {\em Fractal concepts in surface growth}.
\newblock Cambridge University Press, Cambridge, 1995.

\bibitem{MR1071804}
Carol Bezuidenhout and Geoffrey Grimmett.
\newblock The critical contact process dies out.
\newblock {\em Ann. Probab.}, 18(4):1462--1482, 1990.

\bibitem{MR2283880}
B{\'e}la Bollob{\'a}s and Oliver Riordan.
\newblock {\em Percolation}.
\newblock Cambridge University Press, New York, 2006.

\bibitem{CovilleDirrLuckhaus}
J{\'e}r{\^o}me Coville, Nicolas Dirr, and Stephan Luckhaus.
\newblock Non-existence of positive stationary solutions for a class of
  semi-linear {PDE}s with random coefficients.
\newblock {\em Netw. Heterog. Media}, 5(4):745--763, 2010.

\bibitem{DDGHS}
N.~Dirr, P.~W. Dondl, G.~R. Grimmett, A.~E. Holroyd, and M.~Scheutzow.
\newblock Lipschitz percolation.
\newblock {\em Electron. Commun. Probab.}, 15:14--21, 2010.

\bibitem{DDS11}
Nicolas Dirr, Patrick~W. Dondl, and Michael Scheutzow.
\newblock Pinning of interfaces in random media.
\newblock {\em Interfaces Free Bound.}, 13(3):411--421, 2011.

\bibitem{DS12}
Patrick~W. Dondl and Michael Scheutzow.
\newblock Positive speed of propagation in a semilinear parabolic interface
  model with unbounded random coefficients.
\newblock {\em Netw. Heterog. Media}, 7(1):137--150, 2012.

\bibitem{fisher}
Daniel~S. Fisher.
\newblock Sliding charge-density waves as a dynamic critical phenomenon.
\newblock {\em Phys. Rev. B}, 31:1396--1427, Feb 1985.

\bibitem{Fis96}
Michael~E. Fisher.
\newblock Renormalization group theory: its basis and formulation in
  statistical physics.
\newblock In {\em Conceptual foundations of quantum field theory ({B}oston,
  {MA}, 1996)}, pages 89--135. Cambridge Univ. Press, Cambridge, 1999.

\bibitem{GrimmettMarstrand}
G.~R. Grimmett and J.~M. Marstrand.
\newblock The supercritical phase of percolation is well behaved.
\newblock {\em Proc. Roy. Soc. London Ser. A}, 430(1879):439--457, 1990.

\bibitem{Gri99}
Geoffrey Grimmett.
\newblock {\em Percolation}, volume 321 of {\em Grundlehren der Mathematischen
  Wissenschaften [Fundamental Principles of Mathematical Sciences]}.
\newblock Springer-Verlag, Berlin, second edition, 1999.

\bibitem{hiemer}
Geoffrey Grimmett and Philipp Hiemer.
\newblock Directed percolation and random walk.
\newblock In {\em In and out of equilibrium ({M}ambucaba, 2000)}, volume~51 of
  {\em Progr. Probab.}, pages 273--297. Birkh\"auser Boston, Boston, MA, 2002.

\bibitem{MR0101564}
J.~M. Hammersley.
\newblock Percolation processes: {L}ower bounds for the critical probability.
\newblock {\em Ann. Math. Statist.}, 28:790--795, 1957.

\bibitem{KoltonRosso}
Alejandro~B. Kolton, Alberto Rosso, Thierry Giamarchi, and Werner Krauth.
\newblock Dynamics below the depinning threshold in disordered elastic systems.
\newblock {\em Phys. Rev. Lett.}, 97:057001, Aug 2006.

\bibitem{KoplikLevine}
Joel Koplik and Herbert Levine.
\newblock Interface moving through a random background.
\newblock {\em Phys. Rev. B}, 32:280--292, Jul 1985.

\bibitem{MR1048929}
Thomas Kuczek.
\newblock The central limit theorem for the right edge of supercritical
  oriented percolation.
\newblock {\em Ann. Probab.}, 17(4):1322--1332, 1989.

\bibitem{doussal}
Pierre Le~Doussal, Kay~J\"org Wiese, and Pascal Chauve.
\newblock Functional renormalization group and the field theory of disordered
  elastic systems.
\newblock {\em Phys. Rev. E}, 69:026112, Feb 2004.

\bibitem{Menshikov}
M.~V. Menshikov.
\newblock Coincidence of critical points in percolation problems.
\newblock {\em Dokl. Akad. Nauk SSSR}, 288(6):1308--1311, 1986.

\bibitem{NarayanFisher}
Onuttom Narayan and Daniel~S. Fisher.
\newblock Threshold critical dynamics of driven interfaces in random media.
\newblock {\em Phys. Rev. B}, 48:7030--7042, Sep 1993.

\bibitem{MR841669}
C.~M. Newman and L.~S. Schulman.
\newblock One-dimensional {$1/\vert j-i\vert ^s$} percolation models: the
  existence of a transition for {$s\leq 2$}.
\newblock {\em Comm. Math. Phys.}, 104(4):547--571, 1986.

\bibitem{Pisztora}
Agoston Pisztora.
\newblock Surface order large deviations for {I}sing, {P}otts and percolation
  models.
\newblock {\em Probab. Theory Related Fields}, 104(4):427--466, 1996.

\bibitem{PT12}
Serguei Popov and Augusto Teixeira.
\newblock Soft local times and decoupling of random interlacements.
\newblock accepted for publication in the Journal of the European Mathematical
  Society, 2012.

\bibitem{SV05}
V.~Sidoravicius, M.E. Vares, and Centro de~Estudios Avanzados Instituto
  Venezolano~de Investigaciones~Cientificas.
\newblock {\em Interacting Particle Systems: Renormalization and Multi-scale
  Analysis}.
\newblock Asociacion Matematica Venezolana / Instituto Venezolano de
  Investigaciones Cientificas, 2005.

\bibitem{Szn09}
Alain-Sol Sznitman.
\newblock Vacant set of random interlacements and percolation.
\newblock {\em Ann. of Math. (2)}, 171(3):2039--2087, 2010.

\bibitem{TL92}
Lei-Han Tang and Heiko Leschhorn.
\newblock Pinning by directed percolation.
\newblock {\em Phys. Rev. A}, 45:R8309--R8312, Jun 1992.

\bibitem{zeitouni}
Simon Tavar{\'e} and Ofer Zeitouni.
\newblock {\em Lectures on probability theory and statistics}, volume 1837 of
  {\em Lecture Notes in Mathematics}.
\newblock Springer-Verlag, Berlin, 2004.
\newblock Lectures from the 31st Summer School on Probability Theory held in
  Saint-Flour, July 8--25, 2001, Edited by Jean Picard.

\bibitem{Vannimenus}
Jean Vannimenus and Bernard Derrida.
\newblock A solvable model of interface depinning in random media.
\newblock {\em J. Statist. Phys.}, 105(1-2):1--23, 2001.

\end{thebibliography}

\end{document}